\documentclass[12pt,a4paper]{article}

\usepackage{fullpage}
\usepackage{psfrag}
\usepackage{amsmath}
\usepackage{fancybox}
\usepackage{amsfonts}
\usepackage{amsthm}
\usepackage{amssymb}
\usepackage{comment}
\usepackage{graphicx}
\usepackage[normalem]{ulem} 
\usepackage{tikz}
\usepackage{enumerate}
\usepackage{MnSymbol}
\usepackage{todonotes}
\usepackage{bm}

\DeclareMathOperator*{\argmin}{argmin}

\DeclareMathOperator{\col}{col}

\usepackage{algorithm}
\usepackage{setspace}
\usepackage{algpseudocode}

 \theoremstyle{plain}
    \newtheorem{theorem}{Theorem}[section]
    \newtheorem{corollary}[theorem]{Corollary}
    \newtheorem{lemma}[theorem]{Lemma}

 \theoremstyle{definition}
    
    \newtheorem{remark}[theorem]{Remark}
 \theoremstyle{remark}
    \newtheorem{example}[theorem]{Example}
 
 \usepackage[english]{babel}

\usepackage[%
   bookmarks=true,
   pagebackref=false, 
   colorlinks,
   linkcolor=blue,
   anchorcolor=blue,
   citecolor=blue,
   filecolor=blue,
   urlcolor=blue
   ]{hyperref}

\usepackage{bm}
\usepackage{xspace}

\newcommand{\D}{\displaystyle}
\usepackage{color}

\DeclareMathOperator{\Span}{span}

\DeclareMathOperator{\Tol}{TOL}
\newcommand{\cpnorm}{\text{cp}}
\newcommand{\cpinfnorm}{{\cpnorm\!\infty}}

\newcommand{\loc}{\textrm{loc}}
\newcommand{\uno}{\textrm{first}}
\newcommand{\dos}{\textrm{loc}}
\newcommand{\tres}{\textrm{last}}

\newcommand\NN{\mathbb N}

\newcommand\BB{\mathcal B}

\newcommand{\RR}{\mathbb R}

\renewcommand{\Xi}{{\bm\xi}}
\newcommand{\Mu}{{\bm\mu}}
\renewcommand{\nu}{{\bm\nu}}

\usepackage{authblk}
\newcommand{\norm}[1]{{\left\vert\kern-0.25ex\left\vert\kern-0.25ex\left\vert #1 
    \right\vert\kern-0.25ex\right\vert\kern-0.25ex\right\vert}}
\newcommand{\icero}{{i_0}}
\newcommand{\jcero}{{j_0}}
\newcommand{\cc}[1]{\hat{#1}}

\newcommand{\rr}{{\bf r}_\loc}
\newcommand{\jj}{{\bf j}_\loc}

\newcommand\SSS{\mathcal S}
\newcommand\JJJ{\mathcal J}

\newcommand{\EEE}[1]{\mathbb{E}_{\Xi,{#1}}}

\begin{document}

\title{\bf Error analysis for local coarsening \\in univariate spline spaces}
\author{Silvano Figueroa, Eduardo M. Garau \& Pedro Morin}
	\affil{\small Universidad Nacional del Litoral, \\
 Consejo Nacional de Investigaciones Cient\'ificas y T\'ecnicas, \\
 FIQ, Santa Fe, Argentina}

\maketitle

\begin{abstract} 
In this article we analyze the error produced by the removal of an arbitrary knot from a spline function. When a knot has multiplicity greater than one, this implies a reduction of its multiplicity by one unit. In particular, we deduce a very simple formula to compute the error in terms of some neighboring knots and a few control points of the considered spline. Furthermore, we show precisely how this error is related to the jump of a derivative of the spline at the knot. We then use the developed theory to propose efficient and very low-cost local error indicators and adaptive coarsening algorithms. Finally, we present some numerical experiments to illustrate their performance and show some applications.

\end{abstract}


\begin{quote}\small
\end{quote}

{{\bf Acknowledgements.} This work was partially supported by Consejo Nacional de Investigaciones Cient\'ificas y T\'ecnicas through grant PIP 2021-2023 (11220200101180CO), by Agencia Nacional de Promoci\'on Cient\'ifica y Tecnol\'ogica through grant PICT-2020-SERIE A-03820, and by Universidad Nacional del Litoral through grant CAI+D-2020 50620190100136LI. This support is gratefully acknowledged.}

\section{Introduction} \label{Introducción}

Let us consider a univariate polynomial spline space $\SSS$ of a fixed degree $p\in\NN$. It is usual to associate a corresponding knot vector $\Xi$ and a B-spline basis~\cite{DeBoor,Schumi} so that each $s\in\SSS$ is uniquely determined by $n:=\dim \SSS = \#\Xi-(p+1)$ coefficients referred to as control points of the spline function $s$.

In computer aided design and for different practical purposes, it is useful to enrich such a spline space while maintaining the same spline function. This procedure is known as \emph{knot insertion}~\cite{Boehm80, Oslo}, and a closed formula for updating the control points using the so called Oslo Algorithm is well established. 

On the other hand, it is not possible in general to represent a spline exactly in a coarser space. For the case when some knots can be removed without changing the spline function some algorithms have been proposed in~\cite{Tiller92}. Additionally, the general problem of transforming a spline function between B-spline representations on two arbitrary knot
vectors was considered in~\cite[Chapter 5]{GoldmanLycheBook93}, where the transformation was studied for splines belonging to both spline spaces.

As we mentioned above, it is always possible to represent exactly a spline in a finer space, but the reverse procedure is not possible in general without changing the spline.  
It is thus interesting to analyze the problem of obtaining a suitable approximation $\cc{s}\in\cc{\SSS}$ of a spline $s\in\SSS$, where $\cc{\SSS}\subset \SSS$ is a coarser spline space, i.e., a spline space of smaller dimension. 
This problem involves two mains steps. 
\begin{itemize}
    \item First, we need to decide which are the more suitable knots to remove from $\Xi$ in order to define the coarser knot vector $\cc{\Xi}$. 
    \item Then, once the coarser space $\cc{\SSS}$ is determined, we have to define the spline $\cc{s}$ that approximates $s$. 
   
\end{itemize}

 As far as we know, the notion of knot removal in this sense was first studied in~\cite{Handscomb87}, by considering a reversal of the Oslo algorithm.  Although an in-depth analysis of the error incurred in when removing knots from a given spline was not performed in that article, an interesting discussion on the (non-)uniqueness of solution to the minimax problem was presented.

The knot removal process can be considered as the inverse of the knot insertion algorithm. If these operations were reversible, the control point vector after knot removal would be the solution of a linear system associated to the knot insertion matrix. Such a system has no solution in general and it is thus necessary to consider some generalized solutions. A local construction of the control points after a single knot removal is proposed in~\cite{ECK}, where an analysis and a strategy for choosing a good approximation are presented; a study of which knots are more convenient to remove from a given spline is not provided though. 

It is worth noting that finding the best approximation from the coarser space in a given norm can become costly and involve solving a global problem. It would be convenient to have a localized method, which is inexpensive and updates only the coefficients corresponding to knots which are close to the one being removed.

The goal of this article is twofold. 
First, we analyze the error incurred in when removing a single knot from a spline $s\in\SSS$, defined as
\begin{equation}\label{E:intro}
    \EEE{\jcero}(s):= \min_{g\in\cc{\SSS}} \|g - s \|,
\end{equation}
where $\cc{\SSS}$ is the spline space associated to the knot vector $\cc{\Xi}$ obtained from $\Xi$ by removing one knot\footnote{More precisely, the multiplicity of the $\jcero$-th (interior) breakpoint is reduced by one unit.} and $\| \cdot \|$ is a suitable norm in $\SSS$. In particular, we deduce a simple formula for computing such an error and the control points of the spline in $\cc{\SSS}$ attaining the minimum in~\eqref{E:intro}. Notice that the magnitude of the error associated to each knot provides a criterion for deciding which knot is the most convenient to be removed.
Additionally, we also quantify precisely the relationship between the error~\eqref{E:intro} when removing a knot and the jump of the derivative of a suitable order\footnote{The order of this derivative is the order polynomial $p+1$ minus the multiplicity of the knot being considered.} 
of the spline $s$ at such a knot. At this point, it is important to mention that a knot can be \emph{safely} removed from a spline, i.e., leaving it geometrically unchanged, whenever such derivative is continuous at the knot, or equivalently, if the jump of such derivative vanishes. Based on this observation, a criterion involving third derivatives has been already considered for the design of approximation algorithms using cubic splines of maximum smoothness (cf.~\cite{wever1991global,MR1887747}).

Secondly, using the error analysis described above, we propose some algorithms to adaptively remove knots. More specifically, given a tolerance $\Tol >0$ and a spline $s\in\SSS$, the algorithms compute a coarser spline space $\cc{\SSS}\subset \SSS$ and a spline $\cc{s}\in\cc{\SSS}$ satisfying $\|\cc{s} - s\| < \Tol$; where the norm $\|\cdot\|$ considered can be the $L^2$-, the $L^\infty$-, or the $H^1$-norm.

For our presentation we follow some ideas and the notation from~\cite{LM87} and~\cite[Chapter 6]{MorkenPHDthesis}), because the framework used there is adequate for our developments regarding knot removal.

Finally, it is worth mentioning that a data reduction strategy which automatically removes knots has already been presented in~\cite{T.L&K.M}. 
There, the authors assign a weight $w$ to each interior knot of $\Xi$, which encodes a rough measure of the significance of each knot in the representation of the spline $s$. 
These weights approximate certain distance between the spline $s$ and some subspace of $\SSS$. The knots to be removed are then those with the smaller weights.
The proposed algorithm is very interesting and although it removes several knots simultaneously, it involves an internal loop in which the error must be compared to the tolerance after each individual knot removal. 
After removing the knots they compute the best approximation of $s$ as the solution of a global linear least-squares problem, hence this strategy is effective but costly since the main work must be done using all the data. 
In this article, although we remove one knot at a time, the computational cost is negligible because the discrepancy parameter (estimator) depends only on a few data points and the best approximation can be computed by modifying only a small part of the coefficient vector, and furthermore, at each step, only a few estimators and control points have to be recomputed (most of them remain unchanged).

An extension of~\cite{T.L&K.M} to parametric curves and tensor-product B-spline surfaces was presented in~\cite{b_T.L&K.M}. See also the references therein for previous works on knot removal.

This paper is organized as follows. In Section \ref{Intro} we briefly revise the classic theory of spline spaces, including B-spline bases, stability and knot insertion. In Section~\ref{knot removal} we analyze the error for a single knot removal deriving a simple formula for its computation and prove a characterization of such an error in terms of jumps of derivatives. In Section~\ref{S:Algorithms} we use the previous error analysis for defining local coarsening error indicators and develop some efficient and low-cost adaptive knot removal algorithms which guarantee that the error is below a prescribed tolerance. Finally, in Section~\ref{S:Numerical experiments} we explore numerically the performance of the proposed algorithms and illustrate some practical applications.

\section{Review on basics about univariate spline theory} \label{Intro}

Let $[a,b]\subset \RR$ and let 
$Z=\{a=\zeta_1 < \zeta_2 < \dots < \zeta_N = b\}$ be a set of \emph{breakpoints}.
Let $p\in\NN$ be a polynomial degree, which from now on will remain fixed. To each interior breakpoint $\zeta_j$, $j=2,\dots,{N-1}$, we associate a number $m_j$, called multiplicity, such that 
$1\le m_j\le p+1$. Let $\SSS$ be the space of piecewise polynomial functions of degree $\le p$ on $Z$ that have $p-m_j$ continuous derivatives at the breakpoint $\zeta_j$. It is known that $\SSS$ is a vector space of finite dimension $n:=\dim \SSS= p+1+\sum_{j=2}^{N-1} m_j$.      

\subsection{The B-spline bases and their $L^2$-stability}\label{Bases B-splines}

We consider a well known B-spline basis for the spline space $\SSS$,~see e.g.~\cite{DeBoor, Schumi}. In order to define it we need to associate a \emph{knot vector} $\Xi$ which takes into account the polynomial degree $p$, the breakpoints in $Z$ and the corresponding multiplicity, as we explain next.

Let $\Xi:=\{\xi_j\}_{j=1}^{n+p+1}$ be an associated $(p+1)$-basic knot vector, i.e., 
\begin{equation}\label{E:basic knot vector}
    \Xi=\{\xi_1,\dots,\xi_{p+1},\underbrace{\zeta_2,\dots,\zeta_2,}_{m_2 \text{ times}}\dots ,\underbrace{\zeta_{N-1},\dots,\zeta_{N-1},}_{m_{N-1} \text{ times}}\xi_{n+1},\dots,\xi_{n+p+1}\},
\end{equation} 
where $\xi_1 \leq \ldots \leq \xi_{p+1}=\zeta_1$ and $\zeta_N=\xi_{n+1}\leq \ldots\leq \xi_{n+p+1}$. There exists a basis for $\SSS$, called the B-spline basis $\BB:=\{B_{1,p},B_{2,p},\dots,B_{n,p}\}$, where the the $i$-th B-spline $B_i:=B_{i,p}$ is non-negative, uniquely determined by the knots $\{\xi_{i},\dots,\xi_{i+p+1}\}$, and locally supported in $[\xi_{i},\xi_{i+p+1}]$, for $i=1,\dots,n.$ Moreover, they constitute a convex partition of unity in $[a,b]$, i.e.,
 $$\D\sum_{i=1}^n B_i(x) = 1, \quad \forall x \in [a,b].$$

We consider a useful norm in $\SSS$, that we call the $\Xi$-norm, given by  (cf.~\cite{T.L&K.M})
  \begin{equation} \label{norma escalada}
      \| s\|_{\Xi} := 
		\left( \sum_{i=1}^n c_i^2 \frac{\xi_{i+p+1} - \xi_i}{p+1}\right)^{\frac{1}{2}}, 
  \end{equation}
  where ${\bf c} = (c_1,\dots,c_n)^T\in\RR^n$ is the vector of control points of the spline $s\in\SSS$, i.e., $s=\sum_{i=1}^n c_i B_i$. Defining the $n\times n$ diagonal scaling matrix $E_\Xi$ whose elements are
  \begin{equation}\label{E:wi}
      \omega_i :=\left( \frac{\xi_{i+p+1}-\xi_i}{p+1} \right)^\frac{1}{2}, \,\ i=1, \ldots, n,
  \end{equation}
 we have that~\eqref{norma escalada} can be written as
  $ \| s\|_{\Xi}= \| E_\Xi {\bf c} \|_2.$

The following theorem states that the mesh-dependent norm $\| \,\cdot\,\|_{\Xi}$ is equivalent to the standard $L^2[a,b]$-norm in $\SSS$. At this point, it is key to emphasize that the equivalence constant depends on the polynomial degree $p$ but is otherwise independent of the space $\SSS$. A proof of this result can be found in \cite[Theorem 5.2]{DeBOOR1976}, cf. also~\cite[Proposition 5.2]{T.L&K.M}. More recent proofs have been presented in \cite[Theorem 11]{Cetraro} and \cite[Lemma 5]{Cetraro}. 

\begin{theorem}[$L^2$-stability of the B-spline basis] \label{T:estability}\rm
	 There exists a constant $K_p > 0$, which only depends on $p$, such that
	\begin{equation*} \label{estability of B with scaled norm}
		K_p^{-1} \|s\|_\Xi \leq \|s\|_{L^2} \leq \|s\|_\Xi, \qquad \forall s\in \SSS.
	\end{equation*}
 \end{theorem}

\begin{remark}\label{R:estabilidad}
    This result is known as $L^2$-stability of the B-spline basis because it states the equivalence between the $L^2$-norm of a spline and a suitable vector norm of its coordinates in the B-spline basis. Indeed, this result holds for $L^q$-norms, with $1\le q\le \infty$, but we presented Theorem~\ref{T:estability} in this way to simplify the presentation. In particular, the classical $L^\infty$-stability of the B-spline basis reads:
    \begin{equation} \label{estability of Bsplines in Linfty}
		K_p^{-1} \|{\bf c}\|_\infty \leq \|s\|_{L^\infty} \leq \|{\bf c}\|_\infty, \qquad \forall s\in \SSS,
	\end{equation}
 where, as before, ${\bf c}$ denotes the vector of control points of $s$, and $\|{\bf c}\|_\infty = \max_{1\le i \le n}|c_i|$.
\end{remark}

\subsection{Knot insertion} \label{Knot insertion}

In order to analyze precisely the error in the knot removal process, we first briefly recall some facts about knot insertion~\cite{Boehm80}, its reverse operation. The purpose here is just to introduce a notation that is suitable for presenting the error analysis of knot removal in the next section.

Let $\Xi$ be a $(p+1)$-basic knot vector as in~\eqref{E:basic knot vector} and $\SSS$ be the corresponding spline space. Let $\jcero$ be such that $2\le \jcero \le N-1$, so that  $\zeta_\jcero$ is an interior breakpoint in $Z$. Let $\icero:= p+1+\sum_{j=2}^\jcero m_j$ whence, $\xi_\icero=\zeta_\jcero$, $p+2\le \icero\le n$ and $\xi_{\icero-1}\le \xi_\icero <\xi_{\icero+1}$. Let $\cc{\Xi}:=\Xi\setminus\{\xi_\icero\}$, that is,
\begin{equation*}\label{E:knotvectorcoarse}
\cc{\Xi}=\{\cc{\xi}_1,\dots,\cc{\xi}_{n+p}\}=\{\xi_1,\dots,\xi_{\icero-1},\xi_{\icero+1},\dots,\xi_{n+p+1}\}.
\end{equation*}
Therefore, we can regard $\Xi$ as obtained from the $(p+1)$-basic knot vector $\cc{\Xi}$ after \emph{inserting} the knot $\xi_{i_0}$. We denote by $\ell:=m_\jcero-1$ the multiplicity of $\xi_\icero$ in $\cc{\Xi}$. Here, $\ell=0$ means that $\xi_\icero$ is not a knot in $\cc{\Xi}$, and therefore, from $\cc{\Xi}$ to $\Xi$ we have inserted a new breakpoint; whereas if $1\le \ell\le p$ we have increased by $1$ the multiplicity of the knot corresponding to the breakpoint $\zeta_\jcero$.

Let $\cc{\BB}:=\{\cc{B}_1,\cc{B}_2,\dots,\cc{B}_{n-1}\}$ be the B-spline basis associated to $\cc{\Xi}$ and $\cc{\SSS}$ be the spline space spanned by $\cc{\BB}$, so that $\cc{\SSS} \subset \SSS$. Thus, each $\cc{s}\in\cc{\SSS}$ can be expressed by
\begin{equation*}
\cc{s} = \sum_{i=1}^{n-1}\cc{c}_i\cc{B}_i,
\end{equation*}
for some ${\bf \cc{c}}:=(\cc{c}_1,\dots,\cc{c}_{n-1})^T\in\RR^{n-1}$.  Since $\cc{\SSS}\subset \SSS$ there exists ${\bf c}:=(c_1,\dots,c_{n})^T\in\RR^{n}$ such that 
\begin{equation*}
\cc{s} = \sum_{i=1}^{n}c_i B_i.
\end{equation*}
It is well known that the control points of $\cc{s}$ in $\BB$ can be uniquely determined by those in $\cc{\BB}$ through the so called \emph{knot insertion formula}~\cite{Boehm80} as

\begin{equation*}  \label{formula de c}
	c_i= \begin{cases}
		\cc {c_i}, &\text{for} \,\ i=1, \dots ,i_0 -p-2, \\
		\lambda_i \cc{c_i} + (1-\lambda_i) \cc{c}_{i-1}, &\text{for} \,\ i= i_0-p-1, \dots, i_0 -\ell,  \\
		\cc{c}_{i-1}, &\text{for} \,\ i= i_0-\ell+1, \dots, n,
	\end{cases}
\end{equation*}
where
\begin{equation*} 
	\lambda_i= \frac{\xi_\icero - \xi_{i}}{\xi_{i+p+1}-\xi_{i}}, \quad  i= \icero-p-1, \dots, \icero -\ell.
\end{equation*}
Notice that the mapping $\cc{\bf c} \mapsto {\bf c}$ can be expressed in matrix form as 
\begin{equation} \label{global system}
	{\bf c}= A  \cc{\bf c},
\end{equation} 
where the \emph{knot insertion matrix} $A\in\RR^{n\times(n-1)}$ is given by
\begin{equation} \label{Matriz A}
A= \begin{pmatrix}
	I_{\icero -p-2} & 0 & 0\\
	0 & A_{\loc} & 0\\
	0 & 0 & I_{n-\icero+\ell}
\end{pmatrix},
\end{equation}
with the sub-matrix $A_{\loc}\in\RR^{(p+2-\ell)\times(p+1-\ell)}$, hereafter called \emph{local knot insertion matrix}, given by
\begin{equation} \label{Matriz Aloc}
A_{\loc}= \begin{pmatrix}
	\alpha_1 & 0& 0 & \cdots & 0& 0\\
	1-\alpha_2 & \alpha_2 & 0&  \cdots & 0 & 0\\
	0 & 1-\alpha_3 & \alpha_3 &\cdots & 0 & 0\\
	\vdots & \vdots & \vdots& \ddots & \vdots &\vdots\\
	0 & 0 & 0 & \cdots & \alpha_{p-\ell} & 0 \\
	0 & 0 & 0 & \cdots & 1-\alpha_{p+1-\ell} & \alpha_{p+1-\ell} \\
	0 & 0 & 0 & \cdots & 0& 1-\alpha_{p+2-\ell}
\end{pmatrix},
\end{equation}
and 
\begin{equation}\label{E:alpha}
    \D\alpha_j:=\lambda_{\icero-p-2+j}
= \frac{\xi_\icero - \xi_{\icero-p-2+j}}{\xi_{\icero-1+j}-\xi_{\icero-p-2+j}},\qquad j=1, \dots,p+2-\ell.
\end{equation}
Notice that $\{\alpha_j\}_{j=1}^{p+2-\ell}$ is monotonically (non-strictly) decreasing with $\alpha_1 = 1$ and $\alpha_{p+2-\ell} = 0$. Additionally, it is important to notice that the matrix $A_{\loc}$ only depends on $2p+1-\ell$ consecutive knots in $\Xi$, namely
 \begin{equation}\label{E:Xilocstar}
     \Xi^*_{\loc}:=\{\xi_{\icero-p}, \ldots, \xi_{\icero+p-\ell}\}.
 \end{equation}

Throughout the next section, we make use of the following splitting for the control points vectors ${\bf c}$ and $\cc{\bf c}$ that allows us to emphasize the local nature of our developments: 
\begin{equation} \label{c_split}
{\bf c}^T = (\underbrace{c_1, \ldots, c_{\icero-p-2}}_{=:{\bf c}^T_\uno}, \underbrace{c_{\icero-p-1}, \ldots, c_{\icero-\ell}}_{=:{\bf c}^T_\dos}, \underbrace{c_{\icero-\ell+1}, \ldots, c_n}_{=:{\bf c}^T_\tres}),
\end{equation}
and
\begin{equation}\label{hatc_split}
{\cc{\bf c}}^T=(\underbrace{\cc{c}_1, \ldots, \cc{c}_{\icero-p-2}}_{=:{ \cc{\bf c}^T_\uno}}, \underbrace{\cc{c}_{\icero-p-1}, \ldots, \cc{c}_{\icero-\ell-1}}_{=:{ \cc{\bf c}^T_\dos}}, \underbrace{\cc{c}_{\icero-\ell}, \ldots, \cc{c}_{n-1}}_{=:{ \cc{\bf c}^T_\tres}}).
\end{equation}
In particular, notice that~\eqref{global system} means ${\bf c}_\uno = \cc{\bf c}_\uno$, ${\bf c}_\tres = \cc{\bf c}_\tres$ and
\begin{equation*} 
	{\bf c}_\dos= A_{\loc} \cc{\bf c}_\dos.
\end{equation*}

\section{Error analysis for single knot removal} 
\label{knot removal}

Let us consider a $(p+1)$-basic knot vector $\Xi:= \{\xi_1, \dots, \xi_{n+p+1}\}$ associated to a set of breakpoints $Z=\{\zeta_1,\dots,\zeta_N\}$ as explained in the previous section. Let $\SSS$ denote the spline space spanned by the B-spline basis $\BB= \{B_1, \dots, B_n\}$ corresponding to $\Xi$ with $\dim{\SSS}=n$.

Throughout this section we consider a fixed interior breakpoint $\zeta_\jcero$, i.e., $2 \le \jcero \le N-1$. 
The knot removal that we consider, understood as the natural reverse process of single knot insertion, consists of decreasing by one the multiplicity of $\zeta_\jcero$ in the sequence of knots $\Xi$. 
Without loss of generality, we consider the removal of the last knot in $\Xi$ that is equal to $\zeta_{j_0}$, i.e., we remove from $\Xi$ the knot $\xi_{i_0}$, with $\icero:= p+1+\sum_{j=2}^\jcero m_j$, so that $p+2\le \icero\le n$ and $\xi_{\icero-1}\le \xi_\icero =\zeta_\jcero <\xi_{\icero+1}$.

Now, we consider the $(p+1)$-basic knot vector $\cc{\Xi}:= \Xi \setminus \{\xi_\icero\} $ and let $\cc{\SSS}$ be the spline space spanned by the B-spline basis $\cc{\BB}= \{\cc{B}_1, \dots, \cc{B}_{n-1}\}$ corresponding to~$\cc{\Xi}$.

The main purpose of this section is to analyze the error
\begin{equation}\label{E:error}
    \EEE{\jcero}^{\| \cdot \|}(s):=\min_{g\in \cc{\SSS}} \| g-s\|,
\end{equation}
for a given spline $s\in\SSS$, where $\|\cdot\|$ denotes a norm in $\SSS$. A best approximation $\cc{s}\in\cc{\SSS}$ is a spline function attaining the minimum, that is,
\begin{equation}\label{E:best approximation}
\cc{s} := \argmin_{g\in \cc{\SSS}} \| g-s \|.
\end{equation}
Our analysis includes a characterization of the error in~\eqref{E:error} considering the $\Xi$-norm defined in~\eqref{norma escalada} and a derivation of a simple formula for computing it in terms of the neighboring knots of $\xi_{\icero}$ and a few control points of the spline $s$. In addition, we propose an efficient way of computing the best approximation in~\eqref{E:best approximation}.

From now on, the error  $\EEE{\jcero}^{\| \cdot \|_\Xi}(s)$ will be denoted simply by $\EEE{\jcero}(s)$, and with a little abuse of notation, if ${\bf c} :=(c_1,c_2,\dots,c_n)^T\in \RR^n$ denotes the vector of coefficients of $s$ in $\BB$, named \emph{control points}, we denote $\|s\|_\Xi = \|{\bf c}\|_\Xi$. It is then easy to verify that
\begin{equation} \label{E:error matricial}
	\EEE{\jcero}(s)= \min_{{\bf b}\in\RR^{n-1}}\| A{\bf b}-{\bf c}\|_\Xi= \min_{{\bf b}\in\RR^{n-1}}\|E_\Xi( A{\bf b}-{\bf c})\|_2,
\end{equation}
where $E_\Xi$ is the scaled diagonal matrix whose entries are given by~\eqref{E:wi} and $A$ is the knot insertion matrix from $\cc{\Xi}$ to $\Xi$ defined in Section~\ref{Knot insertion}. Moreover, if $\cc{s}$ denotes the best approximation given by~\eqref{E:best approximation}, the vector $(\cc{c}_1,\cc{c}_2\dots, \cc{c}_{n-1})^T\in \RR^{n-1}$ of coefficients of $\cc{s}$ in $\cc{\BB}$ satisfies 
\begin{equation} \label{E:best approximation matricial}
	{\bf \cc{c}}:= \argmin_{{\bf b}\in\RR^{n-1}}\| A{\bf b}-{\bf c}\|_\Xi= \argmin_{{\bf b}\in\RR^{n-1}}\|E_\Xi( A{\bf b}-{\bf c})\|_2.
\end{equation}

We notice that when considering the $L^2$-norm in $\SSS$, the spline $\cc{s}$ in~\eqref{E:best approximation} is given by the $L^2$-projection of $s$ onto $\cc{\SSS}$, denoted by $\Pi_{\cc{\SSS}} s$. The next theorem states that the errors in~\eqref{E:error} are equivalent when considering the $L^2$-norm and the $\Xi$-norm. This result is a particular case from~\cite[Theorem 2.1]{LM87} (see also~\cite[Theorem 6.1]{MorkenPHDthesis}) but we include a proof here for the sake of completeness.

\begin{theorem}\rm \label{T:relation between the minimization problems}
Let $s\in \SSS$. Suppose that $g= \Pi_{\cc{\SSS}} s$ and $\cc{s}$ are the solutions of \eqref{E:best approximation} when we consider $\| \cdot \|=\| \cdot \|_{L^2}$ and $\| \cdot \|=\| \cdot \|_{\Xi}$, respectively, then
	\begin{equation*}
		\|g- s \|_{L^2} \leq \| \cc{s}-s\|_{L^2} \leq K_p \| g-s \|_{L^2},
	\end{equation*}
where $K_p$ is the constant from Theorem~\ref{T:estability}, which depends only on $p$, but is otherwise independent of $\Xi$ and $\cc{\Xi}$.
\end{theorem}

\begin{proof}
Let ${\bf c}=(c_i)_{i=1}^n$ denote the control points of an arbitrary $s \in \SSS$ so that $s=\sum_{i=1}^n c_i B_i$. Let $g=\Pi_{\cc{\SSS}} s=\sum_{i=1}^{n-1} d_i \cc{B}_i \in \cc{\SSS}$ and  $\cc{s} = \sum_{i=1}^{n-1} \cc{c}_i \cc{B}_i \in \cc{\SSS}$  be the solutions of \eqref{E:best approximation} considering the $L^2$- and the $\Xi$-norms, respectively. Therefore, if ${\bf d}=(d_i)_{i=1}^{n-1}$ and ${\bf \cc{c}}=(\cc{c}_i)_{i=1}^{n-1}$, we have
\[
\|g-s\|_{L^2} \le \|  \cc{s} -s\|_{L^2}
\qquad\text{and}\qquad
\| A{\bf \cc{c}}-{\bf c}\|_{\Xi} \le \|A{\bf {d}}-{\bf c}\|_{\Xi},
\]
with $A$ the knot insertion matrix of $\cc{\Xi}$ on $\Xi$. Also, $\cc{s}= \sum_{i=1}^n (A {\bf \cc{c}})_i B_i$ and $g= \sum_{i=1}^n (A {\bf d})_i B_i$, whence Theorem~\ref{T:estability} yields
\[
\| \cc{s}-s \|_{L^2}
\le \| A{\bf \cc{c}}-{\bf c}\|_{\Xi} 
\qquad\text{and}\qquad 
\|A{\bf {d}}-{\bf c}\|_{\Xi} \le K_p \|  g-s\|_{L^2}.
\]
The proof concludes taking into account these four inequalities. 
\end{proof}

In view of the last result, although we are interested in the $L^2$-norm, we will focus on working with the $\Xi$-norm because it has the advantage of being localized in the sense that the computation of~\eqref{E:error matricial} and~\eqref{E:best approximation matricial} involve only a small part of the data as we will see in the next two subsections.

\subsection{Computation of the best approximation in the $\xi$-norm} \label{calculo de la mejor aprox}

Our next goal is a characterization of the solution ${\bf \cc{c}}$ of~\eqref{E:best approximation matricial} which allows its computation without solving a global system. The results of this section have been briefly introduced in~\cite[Example 4.1]{LM87}. In order to make our presentation clearer, we expand them here and state them more precisely using the notation followed in this article.

Recall that $A$ is the knot insertion matrix from $\cc{\Xi}$ to $\Xi$ defined in Section~\ref{Knot insertion} and that $\ell+1$ denotes the multiplicity of the knot $\xi_{i_0}$ in the knot vector $\Xi$.
  Notice that
$$ \|A{\bf b}-{\bf c}\|_{\Xi}=\|E_\Xi (A{\bf b}-{\bf c})\|_{2}= \|E_\Xi A{\bf b}-E_\Xi{\bf c} \|_{2}= \|B {\bf b}- {\bf{d}}  \|_2$$
where $B:=E_\Xi A \in \RR^{n\times (n-1)}$ and ${\bf{d}}:= E_\Xi {\bf c} \in \RR^n$. Thus, problem~\eqref{E:best approximation matricial} has a unique solution $\bf \cc{c}$, which is the least squares solution of the system 
\begin{equation*} \label{E:least squares system}
	B {\bf \cc{c}}= {\bf d}.
\end{equation*}
In order to write the matrix $B$ in blocks, we first remark that matrix $E_\Xi$ can be expressed as 
\begin{equation}\label{E:Matrix E}
    E_\Xi=  \begin{pmatrix}
E_{\uno} &0 & 0\\
0 & E_{\loc} & 0\\
0 &0 & E_{\tres}
\end{pmatrix}
\end{equation}
where $E_\uno, E_\loc$ and $E_\tres$ are diagonal matrices, whose main diagonals are 
\begin{align*}
{\bf e}_\uno &= (\omega_1, \ldots, \omega_{\icero-p-2}) \in \RR^{\icero-p-2},\\
{\bf e}_\loc &= (\omega_{\icero-p-1}, \ldots, \omega_{\icero-\ell}) \in \RR^{p+2-\ell},\\
{\bf e}_\tres &= (\omega_{\icero+1-\ell}, \ldots, \omega_{n}) \in \RR^{n-\icero+\ell},
\end{align*}
with $\omega_j$ given in~\eqref{E:wi}. Notice that $E_{\loc}\in\RR^{(p+2-\ell)\times(p+2-\ell)}$ and defining 
\begin{equation}\label{ej}
    e_j :=\omega_{\icero-p-2+j},\qquad \text{for }j=1,\dots,p+2-\ell,
\end{equation}
we have that ${\bf e}_\loc = (e_1,\dots,,e_{p+2-\ell})$.  
Now, taking into account~\eqref{Matriz A} and that $B=E_\Xi A$ we have that
\begin{equation}\label{E:Matrix B}
    B=  \begin{pmatrix}
E_{\uno} &0 & 0\\
0 & B_{\loc} & 0\\
0 &0 & E_{\tres}
\end{pmatrix},
\end{equation}
with $B_{\loc}:= E_{\loc}A_{\loc}$ and $A_{\loc}$ as in~\eqref{Matriz Aloc}.
It is worth noticing that the matrix $B_\loc$ only depends on $2p+3-\ell$ consecutive knots in $\Xi$, namely
 \begin{equation}\label{E:Xiloc}
     \Xi_{\loc}:=\{\xi_{\icero-p-1}, \ldots, \xi_{\icero+p+1-\ell}\},
 \end{equation}
 because the matrix $A_\loc$ depends on $\Xi_{\loc}^*$ in~\eqref{E:Xilocstar} and $E_\loc$ depends on $\Xi_{\loc}$.

In the next result we establish a formula to compute the solution ${\bf \cc{c}}$ of~\eqref{E:best approximation matricial} and a characterization of the error in~\eqref{E:error matricial}, which although it measures the $\Xi$-distance of $s$ to the whole space $\cc{\SSS}$, can be computed using only \emph{local} information.

\begin{theorem} \rm \label{Teo del sistema y error local}

 Let $s\in\SSS$ and let $\EEE{\jcero}(s)$ be defined by~\eqref{E:error} using the $\Xi$-norm in $\SSS$. Let ${\bf c} =(c_1,c_2,\dots,c_n)^T\in \RR^n$ be the vector of control points of $s$, i.e., $s=\D\sum_{i=1}^n c_i B_i$. Let $\cc{s}\in\cc{\SSS}$ be the best approximation of $s$ in $\cc{\SSS}$ in the $\Xi$-norm defined in~\eqref{E:best approximation}, and let $\cc{\bf c} :=(\cc{c}_1,\cc{c}_2\dots, \cc{c}_{n-1})^T\in \RR^{n-1}$ be the vector of control points of $\cc{s}$, i.e., $\cc{s}= \D\sum_{i=1}^{n-1} \cc{c}_i \cc{B}_{i}$. Let us consider the splitting of ${\bf c}$ and $\cc{\bf c}$ given in~\eqref{c_split} and~\eqref{hatc_split}, respectively. Then, there hold the following assertions:
 \begin{enumerate}[(i)]
     \item $\cc{\bf c}_\uno = {\bf c}_\uno$, $\cc{\bf c}_\tres = {\bf c}_\tres$ and $\cc{\bf c}_\dos\in\RR^{p+1-\ell}$ is the least squares solution of the system
\begin{equation*} 
	B_{\loc} \cc{\bf c}_\dos = E_\loc{\bf c}_\dos.
 \end{equation*}
 \item The error $\EEE{\jcero}(s)$ satisfies
 
 \begin{equation*}
     \EEE{\jcero}(s) = \|B_{\loc} \cc{\bf c}_\dos - E_\loc{\bf c}_\dos\|_2.
 \end{equation*}
 \end{enumerate}
\end{theorem}

\begin{proof} 
Taking into account~\eqref{E:Matrix B} and~\eqref{E:Matrix E} we have that
$$
			B^T B =
\begin{pmatrix}
				E_{\uno}^2 & 0 & 0\\
				0 & B_{\loc}^T B_\loc & 0\\
				0 & 0 & E_{\tres}^2
			\end{pmatrix} ,
   \quad\text{and}\quad
 B^TE_\Xi  =
   \begin{pmatrix}
				E_{\uno}^2 & 0 & 0\\
				0 & B_{\loc}^T E_\loc& 0\\
				0 & 0 & E_{\tres}^2
			\end{pmatrix}.
		$$
Since $B^TB \cc{\bf c}= B^TE_\Xi{\bf c}$ we have that
$$\begin{cases}
			E_{\uno}^2 \cc{{\bf c}}_{\uno}= E_{\uno}^2 {\bf c}_{\uno} \\[5pt]
			B_{\loc}^T B_\loc\cc{{\bf c}}_\dos= B_{\loc}^T E_\loc {\bf c}_\dos, \\[5pt]
			E_{\tres}^2 \cc{{\bf c}}_{\tres}= E_{\tres}^2 {\bf c}_{\tres} 
		\end{cases}\quad\text{i.e.}\quad \begin{cases}
			\cc{{\bf c}}_{\uno}= {\bf c}_{\uno} \\[5pt]
			B_{\loc}^T B_\loc\cc{{\bf c}}_\dos= B_{\loc}^T E_\loc {\bf c}_\dos ,
    \\[5pt]
			\cc{{\bf c}}_{\tres}=  {\bf c}_{\tres}
		\end{cases}$$
which implies the first assertion of the theorem.

Additionally, we have that
		$$B {\bf \cc{{\bf c}}}- E_\Xi {\bf c} = \begin{pmatrix}
			E_{\uno}\cc{{\bf c}}_{\uno} \\[5pt]
			B_{\loc} \cc{{\bf c}}_\dos \\[5pt]
			E_{\tres} \cc{{\bf c}}_{\tres} 
		\end{pmatrix} - 
		\begin{pmatrix}
			E_{\uno} {\bf c}_{\uno} \\[5pt]
			E_{\loc} {\bf c}_{\dos} \\[5pt]
			E_{\tres} {\bf c}_{\tres} 
		\end{pmatrix} =
		\begin{pmatrix}
			0 \\[5pt]
			B_{\loc} \cc{{\bf c}}_\loc - E_{\loc} {\bf c}_\loc \\[5pt]
			0
		\end{pmatrix}, $$
		and therefore,
		$$\EEE{\jcero}(s) = \| B {\bf \cc{c}} - E_\Xi {\bf c} \|_2= \| B_{\loc} \cc{{\bf c}}_{\dos} - E_{\loc} {\bf c}_\dos \|_2,$$
  which is the second assertion.
\end{proof}

\begin{remark}\label{R:local problem}
	Since $B_{\loc}= E_{\loc}A_{\loc}$, the assertions in Theorem~\ref{Teo del sistema y error local} can be stated as
 \begin{equation*}
     \EEE{\jcero}(s) =\min_{{\bf x}\in\RR^{p+1-\ell}}
  \|E_{\loc}(A_\loc {\bf x} - {\bf c}_\dos)\|_2,
 \end{equation*} 
  and
	$$ \cc{\bf c}_\dos =\argmin_{{\bf x}\in\RR^{p+1-\ell}}
  \|E_{\loc}(A_\loc {\bf x} - {\bf c}_\dos)\|_2.$$
\end{remark}

\subsection{A local formula for the error in $\xi$-norm}

In this section we find a simple formula for computing the error $\EEE{\jcero}(s)$ in terms of $\Xi_{\loc}$ defined in~\eqref{E:Xiloc} and ${\bf c}_\dos=(c_{\icero-p-1}, \ldots, c_{\icero-\ell})^T$.

We need the following result, whose proof follows from elementary linear algebra.

\begin{lemma}\label{L:QR}
\rm Let $M\in\RR^{(k+1)\times k}$ be a matrix with linear independent columns. If ${\bf q}\in\RR^{k+1}$ satisfies $\|{\bf q}\|_2=1$ and ${\bf q}^TM = 0$, then
$$\min_{{\bf x}\in\RR^{k}}
  \|M {\bf x} - {\bf y}\|_2=|{\bf q}^T{\bf y}|,$$
  for all ${\bf y}\in\RR^{k+1}$.
\end{lemma}

\begin{proof} 
Since $M$ has $k$ linearly independent columns, the column space $C(M)$ of $M$ has dimension $k$ and its orthogonal complement $C(M)^\perp$ has dimension one. 
Therefore, if ${\bf q}\in\RR^{k+1}$ satisfies $\|{\bf q}\|_2=1$ and ${\bf q}^TM = 0$ we have that $C(M)^\perp=\Span\{{\bf q}\}$.

 Given ${\bf y}\in\RR^{k+1}$, the minimum of $\|M {\bf x} - {\bf y}\|_2$ is achieved when $M{\bf x}-{\bf y}$ is orthogonal to $C(M)$, whence ${\bf y}-M{\bf x}$ is  the orthogonal projection of ${\bf y}$ onto $C(M)^\perp$. Therefore 
 \begin{equation*}
 \min_{{\bf x}\in\RR^{k}}
  \|M {\bf x} - {\bf y}\|_2
  =\|({\bf q}^T{\bf y}){\bf q}\|_2
  = |{\bf q}^T{\bf y}| \, \|{\bf q}\|_2 
  = |{\bf q}^T{\bf y} |.
\end{equation*}
\end{proof}

Taking into account Remark~\ref{R:local problem} and the last lemma, we are now in position to establish a formula for computing the error.

\begin{theorem}[Main result I]\label{T:error formula}
\rm
Let $\Mu = (\mu_1,\dots,\mu_{p+1-\ell})^T\in\RR^{p+1-\ell}$ be defined by
\begin{equation}\label{E:mu}
    \mu_{j-1}:=\frac{1-\alpha_j}{\alpha_{j-1}},\qquad j=2,\dots,p+2-\ell,
\end{equation}
where $\{\alpha_j\}_{j=1}^{p+2-\ell}$ is the set of values defining the local knot insertion matrix given by~\eqref{E:alpha}. Let $\rr:=(r_1,\dots,r_{p+2-\ell})^T\in\RR^{p+2-\ell}$ be defined by $r_{p+2-\ell}=\gamma_\loc e_{p+2-\ell}$ and
\begin{equation} \label{r: recurrence formula}
    r_{j-1}:=-\mu_{j-1}r_j,\qquad j=2,\dots,p+2-\ell,
\end{equation} 
where $\gamma_\loc:= \left( 1+ e_{p+2-\ell} \sum _{j=1}^{p+1-\ell} \frac{1}{e_j^2} \prod_{i=j}^{p+1-\ell} \mu_i^2 \right)^{-\frac{1}{2}}$, and $e_j$ is given in~\eqref{ej}.
Then, for all $s\in\SSS$,
$$\EEE{\jcero}(s) = |\rr^T {\bf c}_\dos|,$$
where $s=\sum_{i=1}^n c_i B_i$ and ${\bf c}_\dos=(c_{\icero-p-1}, \ldots, c_{\icero-\ell})^T$.
\end{theorem}

\begin{remark}\label{R:calculo del error}
   We recall that $\EEE{\jcero}(s)$ denotes the error of the best approximation of a spline $s\in\SSS$ when reducing by one unit the multiplicity of the $j_0$-th breakpoint. This theorem shows that the error $\EEE{\jcero}:\SSS\to\RR_+$ can be characterized by the vector $\rr\in\RR^{p+2-\ell}$, which can be easily computed from the $2p+3-\ell$ consecutive knots collected in $\Xi_{\loc}=\{\xi_{\icero-p-1}, \ldots, \xi_{\icero+p+1-\ell}\}$.
\end{remark}

The vector $\rr$ proposed in the statement of this theorem is obtained by setting a value for the last component $r_{p+2-\ell}$ and performing backward substitution on the upper triangular bidiagonal matrix $A_\loc^T$. 

\begin{proof}
From Remark~\ref{R:local problem} we have that
$$\EEE{\jcero}(s)=\min_{{\bf x}\in \RR^{p+1-\ell}} \| E_\loc A_\loc {\bf x} - E_\loc {\bf c}_\loc \|_2.$$
We will apply Lemma \ref{L:QR} with $M=E_\loc A_\loc$ and ${\bf y} =E_\loc{\bf c}_\loc$. Let ${\bf q}_\loc:= E_\loc^{-1} {\bf r}_\loc$ and let us prove that ${\bf q}_\loc$ is orthogonal to the columns of the matrix $M$, i.e.
${\bf q}_\loc^T E_\loc A_\loc=0$, or equivalently, that
$${\bf r}_\loc^T A_\loc=0.$$
Indeed, given $j\in \{ 1, \ldots, p+1-\ell\}$, taking into account~\eqref{Matriz Aloc}, we have that
$$ ({\bf r}_\loc^T A_\loc)_j= {\bf r}_\loc^T \col_j (A_\loc)= \alpha_j r_j + (1-\alpha_{j+1}) r_{j+1},$$
and since $r_j= -\mu_j r_{j+1}= - \frac{1-\alpha_{j+1}}{\alpha_j}r_{j+1}$,
we have
$$({\bf r}_\loc^T A_\loc)_j=-\alpha_j \frac{1-\alpha_{j+1}}{\alpha_j}r_{j+1} + (1-\alpha_{j+1})r_{j+1} = -(1-\alpha_{j+1})r_{j+1} + (1-\alpha_{j+1})r_{j+1} =0.$$ 

Since the definition of $\gamma_\loc$ guarantees that $\|{\bf q}_\loc \|_2=1$, we can finally apply Lemma~\ref{L:QR} to obtain
$$\EEE{\jcero}(s)=| {\bf q}_\loc^T {\bf y} |=| (E_\loc^{-1}{\bf r}_\loc)^T (E_\loc {\bf c}_\loc) |= |{\bf r}_\loc^T {\bf c}_\loc |,$$
which concludes the proof.
\end{proof}

\begin{remark}\label{R:Other indicators}
 Notice that $\| \cdot \|_\cpnorm : \SSS \to \RR$, defined by
  $$ \| s\|_{\cpnorm}= \| {\bf c} \|_2= \Big( \sum_{i=1}^n c_i^2 \Big)^{\frac{1}{2}},$$
  is a norm in $\SSS$. Following the same argument from the last section we have that
 \begin{equation*}
     \EEE{\jcero}^{\| \cdot \|_{\cpnorm}}(s) =\min_{{\bf x}\in\RR^{p+1-\ell}}
  \|A_\loc {\bf x} - {\bf c}_\dos\|_2.
 \end{equation*}
  Moreover, the same steps of the proof of Theorem~\ref{T:error formula} imply that this error can be computed by
  \begin{equation}\label{E:error en norma sin peso}
     \EEE{\jcero}^{\| \cdot \|_{\cpnorm}}(s) =|\tilde{\bf r}^T_\loc {\bf c}_\loc|,
 \end{equation}
with the last component of $\tilde{\bf r}_\loc \in \RR^{p+2-\ell}$ given by~$\tilde{r}_{p+2-\ell}:=  \tilde {\gamma}_\loc:= \left( 1+  \sum _{j=1}^{p+1-\ell}  \prod_{i=j}^{p+1-\ell} \mu_i^2 \right)^{-\frac{1}{2}}$, and the remaining components given by the recurrence formula \eqref{r: recurrence formula}. 
On the other hand, we remark that the value $D = D_{\Xi,j_0}(s)$ introduced in \cite[Equation 24]{ECK} satisfies $|D|=|{\bf d}_\loc^T{\bf c}_\loc|$, where the last component of the vector ${\bf d}_\loc \in \RR^{p+2-\ell}$ is set as~${d}_{p+2-\ell}:=1,$ and the remaining components are given by the recurrence formula \eqref{r: recurrence formula}. It is worth noting that the three quantities $\EEE{\jcero}(s)$, $\EEE{\jcero}^{\| \cdot \|_{\cpnorm}}(s)$, and $|D|$, are the absolute value of a linear functional from $\SSS$ into $\RR$ that vanishes on $\cc{\SSS}$, the dimension of which is one smaller than that of $\SSS$. This explains why they can all be computed as the scalar product with vectors $\bf r$, $\bf \tilde r$, $\bf d$, respectively, which are parallel.

In particular, 
\begin{equation*}
    \delta_\loc\EEE{\jcero}^{\| \cdot \|_{\cpnorm}}(s)=  \EEE{\jcero}(s)= \beta_\loc |D_{\Xi,j_0}(s)|, \qquad\forall\, s\in\SSS,
\end{equation*}
with $ \delta_\loc= \frac{r_{p+2-\ell}}{\tilde{r}_{p+2-\ell}}$ and $\beta_\loc= r_{p+2-\ell}$.

It is also worth noticing that even though these three indicators are multiples of each other, the factors relating them depend on the knots in $\Xi_\loc$ and therefore, such a factors are expected to be different for each breakpoint $\zeta_{\jcero}$. In Section~\ref{S:Numerical experiments} we compare the behavior of coarsening algorithms based on these three estimators. We emphasize here that even though these three quantities vanish if and only if a knot can be safely removed without modifying the spline $s$, the quantities $\EEE{\jcero}(s)$ and $\EEE{\jcero}^{\| \cdot \|_{\cpnorm}}(s)$ also have information about some notion of the error that will be produced by such a knot removal, whereas the quantity $D$ does not.

\end{remark}

\subsection{On the jump of derivatives of spline functions} \label{Saltos de una spline}

In this section we deduce a formula for the jump of the derivatives of a spline function $s$ at the breakpoint $\zeta_\jcero$, which we relate, in the next section, with the error analyzed above.

The jump $\JJJ_{\xi}$ of a piecewise continuous function $s$ is defined by
\begin{equation*}
\JJJ_{\xi}(s) := \lim_{x\to\xi^+} s(x)-  \lim_{x\to\xi^-} s(x).
\end{equation*}

The next result can be found in~\cite[Lemma 3.21]{L&M_draft}; we include a proof here for the sake of completeness.

\begin{lemma} \label{Lemma Jumps for B splines}\rm
	Let $\Xi=\{\xi_j\}_{j=1}^{n+p+1}$ be a $(p+1)$-basic knot vector and let $\BB=\{B_1,B_2,\dots,B_n\}$ be the B-spline basis of degree $p$ associated to $\Xi$. Let $j$ be fixed such that $1\le j\le n$. Let $\xi$ be a knot from  $\xi_j, \ldots, \xi_{j+p+1}$, and let $m$ be its multiplicity among the knots $\xi_j, \ldots, \xi_{j+p+1}$. Then the $(p-m+1)$-th derivative of $B_j$ has a nonzero jump at $\xi$ given by
	\begin{equation} \label{salto en L&M}
		\JJJ_{\xi} \left( D^{p-m+1} B_{j} \right)= \frac{p!}{(m-1)!} \frac{(\xi_{j+p+1}-\xi_j)}{\D \prod_{k=j, \xi_k  \neq \xi}^{j+p+1}(\xi_k - \xi)} .
	\end{equation}
\end{lemma}

\begin{proof}
We will proceed by induction on the degree $p$ so we use a second subscript to make explicit the degree of the B-spline functions. Notice that the multiplicity $m$ satisfies $1\le m\leq p+1$. It is easy to verify that~\eqref{salto en L&M} holds for the case $m=p+1$, using the fact that $ \JJJ_\xi (B_{j,p})$ is equal to $1$ when $\xi=\xi_{j}$ and equal to $-1$ when $\xi=\xi_{j+p+1}$.

Thus, in particular, equality (\ref{salto en L&M}) holds when $p=0$ and $m=1$.

Now, suppose that  (\ref{salto en L&M}) holds for B-splines of degree $p-1$. Taking into account the recurrence formula for the derivative of B-splines (see e.g.~\cite[Theorem 3]{Cetraro}), and the linearity of $D^{r-1}$ for $r\ge 1$ and $\JJJ_{\xi}$, we have 
\begin{equation*}
\JJJ_{\xi} (D^r B_{j,p})= p \left( \frac{\JJJ_\xi (D^{r-1} B_{j,p-1})}{(\xi_{j+p}-\xi_j)} - \frac{\JJJ_\xi (D^{r-1} B_{j+1,p-1})}{(\xi_{j+p+1}-\xi_{j+1})} \right).
\end{equation*}
Taking $r=p-m+1$ we have that

\begin{equation} \label{recurrence jump}
    \JJJ_{\xi} (D^{p-m+1} B_{j,p})= p \left( \frac{\JJJ_\xi (D^{p-m} B_{j,p-1})}{(\xi_{j+p}-\xi_j)} - \frac{\JJJ_\xi (D^{p-m} B_{j+1,p-1})}{(\xi_{j+p+1}-\xi_{j+1})} \right).
\end{equation}
Since $\xi$ is one of the knots from $\xi_j, \ldots, \xi_{j+p+1}$, there are now three possible cases: $\xi=\xi_j$, $\xi=\xi_{j+p+1}$ or $\xi_j < \xi < \xi_{j+p-1}$.

\noindent$\bullet$ Case $\xi=\xi_j$. We observe that $\xi$ occurs $m-1$ times among the knots that define $B_{j+1,p-1}$, and $B_{j+1,p-1}$ has $p-m$ continuous derivatives in $\xi$, whence $\JJJ_{\xi}(D^{p-m} B_{j+1,p-1})=0$.
	Thus, applying the inductive assumption, we have
	\begin{align*}
	    \JJJ_{\xi}(D^{p-m+1} B_{j,p})&= p \frac{\JJJ_{\xi}(D^{p-m} B_{j,p-1})}{\xi_{j+p}-\xi_j}= p \frac{(p-1)!}{(m-1)!} \frac{(\xi_{j+p}-\xi_j)}{\D\prod_{k=j, \xi_k  \neq \xi}^{j+p}(\xi_k-\xi)} \frac{1}{(\xi_{j+p}-\xi_j)}\\
     &= \frac{p!}{(m-1)!} \frac{1}{\D\prod_{k=j, \xi_k  \neq \xi}^{j+p}(\xi_k-\xi)}
     =  \frac{p!}{(m-1)!} \frac{\xi_{j+p+1}-\xi}{\D\prod_{k=j, \xi_k  \neq \xi}^{j+p+1}(\xi_k-\xi)},
	\end{align*}
 which coincides with (\ref{salto en L&M}) because $\xi = \xi_j$.

\noindent$\bullet$ Case $\xi=\xi_{j+p+1}$. A similar argument implies the assertion.
	
 \noindent$\bullet$ Case $\xi_j < \xi < \xi_{j+p+1}$. Now, $\xi$ is a knot of multiplicity $m$ among the knots that define both $B_{j,p-1}$ and $B_{j+1,p-1}$. Applying (\ref{recurrence jump}) and the induction hypothesis, we then obtain
	\begin{align*}
\JJJ_{\xi}(D^{p-m+1} B_{j,p})&= \frac{p!}{(m-1)!} \left( \prod_{k=j, \xi_k  \neq \xi}^{j+p} \frac{1}{(\xi_k -\xi)} - \prod_{k=j+1, \xi_k  \neq \xi}^{j+p+1} \frac{1}{(\xi_k -\xi)}  \right)\\
	&= \frac{p!}{(m-1)!} \prod_{k=j+1, \xi_k  \neq \xi}^{j+p} \frac{1}{(\xi_k -\xi)} \left( \frac{1}{(\xi_j -\xi)} - \frac{1}{(\xi_{j+p+1}-\xi)}\right)\\
	&=\frac{p!}{(m-1)!} \frac{(\xi_{j+p+1} -\xi_j)}{\D \prod_{k=j, \xi_k  \neq \xi}^{j+p+1}(\xi_k - \xi)}.
	\end{align*}
	which completes the proof.
\end{proof}

As a consequence of the last lemma we can derive a simple and \emph{local} formula for computing the jump of the $(p-\ell)$-th derivative of $s\in\SSS$ at the breakpoint $\zeta_{\jcero}$.

\begin{theorem}\label{T:jump formula} \rm
Let $\SSS$ be a spline space and let $\Xi:= \{\xi_1, \dots, \xi_{n+p+1}\}$ be an associated $(p+1)$-basic knot vector. Let $\zeta_\jcero$ be an interior breakpoint and $\icero$ be the index such that $\zeta_\jcero = \xi_\icero$ and $\xi_{\icero-1}\le \xi_\icero <\xi_{\icero+1}$. Let $\ell+1$ denote the multiplicity of $\xi_{i_0}$ in the knot vector $\Xi$ and let $\jj :=(z_1,\dots,z_{p+2-\ell})^T\in\RR^{p+2-\ell}$ be defined by
\begin{equation}\label{E:z}
   z_{j-\icero+p+2}:= \frac{p!}{\ell!} \frac{(\xi_{j+p+1}-\xi_j)}{\D \prod_{k=j, \xi_k  \neq \xi_\icero}^{j+p+1}(\xi_k - \xi_\icero)},
\qquad 
j= \icero-p-1, \ldots, \icero-\ell. 
\end{equation}
Then, if $s\in\SSS$, and $s=\sum_{i=1}^n c_i B_i$, there holds
$$\JJJ_{\zeta_{\jcero}}(D^{p-\ell}s) = \jj^T {\bf c}_\dos,$$
where ${\bf c}_\dos=(c_{\icero-p-1}, \ldots, c_{\icero-\ell})^T$.
\end{theorem}

Notice that $\jj$ depends only on $\Xi_\loc$ in~\eqref{E:Xiloc} and the polynomial degree $p$. 

\begin{proof}
Let $\{\xi_j,\cdots,\xi_{j+p+1}\}$ be the local knot vector of the $j$-th B-spline $B_j$, for $j=1,\dots,n$. Since $\ell+1$ is the multiplicity of $\xi_{i_0}$ in $\Xi$, we have that $\xi_{i_0-\ell-1}<\xi_{i_0-\ell}=\dots =\xi_{i_0}<\xi_{i_0+1}$. Thus, $\xi_{i_0}$ appears $m=\ell+1$ times in the local knot vector of $B_j$ provided $j= i_0-p-1,\dots,i_0-\ell$. Otherwise, the multiplicity of $\xi_{i_0}$ among the knots $\xi_j,\cdots,\xi_{j+p+1}$ is at most $\ell$ and in consequence, $D^{p-\ell}B_j$ is continuous at $\xi_{i_0}$.
Now, if $s=\sum_{j=1}^{n} c_j B_j\in\SSS$, using (\ref{salto en L&M}) we have that
\begin{equation*} \label{Salto de la derivada de orden p menos l}
	\JJJ_{\zeta_\jcero} (D^{p-\ell} s)=\JJJ_{\xi_\icero} (D^{p-\ell} s)= \sum_{j=\icero-p-1}^{\icero-\ell} c_j \JJJ_{\xi_\icero}(D^{p-\ell} B_j)= \sum_{j=\icero-p-1}^{\icero-\ell} \frac{p!}{\ell!} \frac{(\xi_{j+p+1}-\xi_j)}{\D \prod_{k=j, \xi_k  \neq \xi_\icero}^{j+p+1}(\xi_k - \xi_\icero)} c_j,
\end{equation*}
which is the desired assertion.
\end{proof}

\subsection{Relationship between the error and the jump}
\label{S:error-vs-jump}

Using the result from the two previous subsections, we now state the precise connection between the error $\EEE{\jcero}(s)$ defined in~\eqref{E:error} and the jump $\JJJ_{\zeta_\jcero}(D^{p-\ell}s)$.

\begin{theorem}[Main result II]\label{T:error y salto} \rm

Let $\SSS$ be a spline space and let $\Xi:= \{\xi_1, \dots, \xi_{n+p+1}\}$ be an associated $(p+1)$-basic knot vector. Let $\zeta_\jcero$ be an interior breakpoint and $\icero$ be the index such that $\zeta_\jcero = \xi_\icero$ and $\xi_{\icero-1}\le \xi_\icero <\xi_{\icero+1}$. Let $\ell+1$ denote the multiplicity of $\xi_{i_0}$ in the knot vector $\Xi$. If $s\in\SSS$,
\begin{equation}\label{E:error-residuo}
	\EEE{\jcero}(s)= C_\loc |\JJJ_{\zeta_\jcero}(D^{p-\ell} s)|,
\end{equation}
where $C_\loc$ is a positive constant which depends only on $\Xi_\loc$ in~\eqref{E:Xiloc}, defined explicitly by 
\begin{equation}\label{Cloc}
	C_\loc:=r_{p+2-\ell} \frac{ \ell!}{ p!} \prod_{k=\icero+1}^{\icero+p-\ell} (\xi_k -\xi_{\icero}),
\end{equation}
where $r_{p+2-\ell}$ is the constant from the statement of Theorem~\ref{T:error formula}.
\end{theorem}

\begin{proof}
In view of Theorems~\ref{T:error formula} and~\ref{T:jump formula}, to establish the equality in~\eqref{E:error-residuo} it will be enough to prove that
$$\rr^T=C_\loc \jj^T.$$
On the one hand, using the recursive formula for $r_j$ in~\eqref{r: recurrence formula} we have that
\begin{equation}\label{E:aux rj}
r_j =  r_{p+2-\ell}(-1)^{p-\ell-j} \prod_{i=j}^{p+1-\ell}\mu_i, 
\end{equation}
for $j = 1,\dots, p+2-\ell$, and taking into account~\eqref{E:mu} and~\eqref{E:alpha},
\begin{equation}\label{E:productoria de los mu}
    \prod_{i=j}^{p+1-\ell} \mu_i=\frac{1}{\alpha_j} \prod_{i=j+1}^{p+1-\ell} \frac{1-\alpha_i}{\alpha_i}=\frac{1}{\alpha_j} \prod_{i=j+1}^{p+1-\ell} \frac{\xi_{\icero-1+i}-\xi_\icero}{\xi_\icero - \xi_{\icero-p-2+i}}= (-1)^{p-\ell-j+1}\frac{1}{\alpha_j} \frac{\D\prod_{k=\icero+j}^{\icero+p-\ell}(\xi_k-\xi_\icero)}{\D\prod_{k=\icero-p+j-1}^{\icero-\ell-1}(\xi_k-\xi_\icero)} .
\end{equation} 
On the other hand, since $\xi_{i_0-\ell-1}<\xi_{i_0-\ell}=\dots =\xi_{i_0}<\xi_{i_0+1}$, from~\eqref{E:z} we obtain
\begin{equation}\label{E:aux zj}
    \frac{1}{z_j}= -\alpha_j\frac{\ell!}{p!}  \D\prod_{k=\icero-p-1+j}^{\icero-1+j} (\xi_k -\xi_\icero)= -\alpha_j\frac{\ell!}{p!} \D\prod_{k=\icero-p-1+j}^{\icero-\ell-1} (\xi_k -\xi_\icero)\D\prod_{k=\icero+1}^{\icero-1+j} (\xi_k -\xi_\icero).
\end{equation}
Finally,~\eqref{E:aux rj},~\eqref{E:productoria de los mu} and~\eqref{E:aux zj} imply that $\D\frac{r_j}{z_j}= C_\loc$, for $j = 1,\dots, p+2-\ell$,
which completes the proof.
\end{proof}

If we focus on the $L^2$-norm, using Theorems~\ref{T:relation between the minimization problems} and~\ref{T:estability} and the last theorem we conclude the following result.

\begin{corollary}\rm

Under the assumptions of Theorem~\ref{T:error y salto}, let $\cc{\SSS}$ be the spline space associated to $\cc{\Xi}:= \Xi \setminus \{\xi_\icero\}$. Then, for any $s\in\SSS$, 
	\begin{equation}\label{E:main}
	\min_{\cc{s}\in\cc{\SSS}}\|s- \cc{s}\|_{L^2} \le C_\loc |\JJJ_{\zeta_\jcero}(D^{p-\ell}s)|,
	\end{equation}
 with $C_\loc$ as in~\eqref{Cloc}.
	
\end{corollary}
We remark that although the right hand side in~\eqref{E:main} is fully computable, this quantity actually excessively overestimates the $L^2$-error in general.

\section{Algorithms for adaptive knot removal}\label{S:Algorithms}

Theorem~\ref{T:error formula} provides a simple formula to compute the error when removing a single knot from a spline. In this section we use that formula as an error indicator and develop an adaptive algorithm for coarsening. More precisely, if $s$ is a spline and $\Tol>0$ is a prescribed tolerance, the algorithm computes a spline $\cc{s}$, which belongs to a coarser spline space, such that $\|s-\cc{s}\|\le \Tol$. We consider some norms that are important in applications such as the $L^2$-, the $L^\infty$- and the $H^1$-norm.   

\subsection{Algorithm for the $L^2$- and the $\xi$-norm}

Algorithm~\ref{alg:L2adaptive knots removal} is our proposed adaptive coarsening algorithm for knot removal up to a tolerance. It starts with a prescribed value $\Tol > 0$ and a spline $s \in \SSS$, where $\SSS$ denotes the spline space associated to a given $(p+1)$-basic knot vector $\Xi$. Its goal is to remove a large number of knots and find a new spline in this coarser spline space, such that the distance between the original spline and the final one, in ${L^2}$- as well as in ${\Xi}$-norm, is less than $\Tol$.

\algdef{SE}[SUBALG]{Indent}{EndIndent}{}{\algorithmicend\ }%
\algtext*{Indent}
\algtext*{EndIndent}

\begin{algorithm}[htb]
\setstretch{1.35}
	\caption{\bf $L^2$\_adaptive\_knot\_removal}  \label{alg:L2adaptive knots removal}
	\begin{algorithmic}[1]
		\Statex \textbf{Input:} $\Tol>0$,  ${\bf c}$ and $\Xi$ (Here, ${\bf c}$ contains the B-spline coefficients of $s\in \SSS$, where $\SSS$ is the spline space associated to a $(p+1)$-basic knot vector~$\Xi$) 
  \State Let $k_{\max}= \# \Xi - 2(p+1)$ be the number of interior knots.
  \State $\{ \varepsilon_j^{(0)}\}_{j=2}^{N-1}\gets \text{\bf compute\_local\_indicators} ({\bf c}, \Xi)$ \Comment{$N$: number of breakpoints of $\Xi$}
		\State $j_*^{(0)} \gets 
  \arg\min_{2\le j\le N-1} \{ \varepsilon_j^{(0)}\}$
  \State $k=0$, $N^{(0)}=N$, ${\bf c}^{(0)} = {\bf c},\, \Xi^{(0)}=\Xi$ 
		\State While ($\varepsilon_{j_*^{(k)}}^{(k)} <$ TOL and $k< k_{\max}$)
  \Indent
  \State TOL $\gets \text{TOL}-\varepsilon_{j_*^{(k)}}^{(k)}$
		\State ${\bf c}_\loc^{(k)} \gets \text{\bf extract\_local\_control\_points} ({\bf c}^{(k)}) $
        \State ${\bf c}_\loc^{(k+1)} \gets \text{\bf compute\_new\_local\_control\_points} ({\bf c}_\loc^{(k)})$
		\State ${\bf c}^{(k+1)} \gets \text{\bf compute\_new\_control\_points} ({\bf c}^{(k)},{\bf c}_{\loc}^{(k+1)})$
            \State $[\Xi^{(k+1)},N^{(k+1)}] \gets \textbf{update\_knot\_vector}(\Xi^{(k)})$
            \State $\{ \varepsilon_j^{(k+1)}\}_{j=2}^{N^{(k+1)}-1}\gets \text{\bf update\_local\_indicators} (\{ \varepsilon_j^{(k)}\}_{j=2}^{N^{(k)}-1})$
		\State $k=k+1$
		\State $j_*^{(k)} \gets 
  \arg\min_{2\le j\le N^{(k)} -1} \{ \varepsilon_j^{(k)}\}$
\EndIndent
\State end While
\State $\cc{\Xi}=\Xi^{(k)}$, $\cc{\bf c}= {\bf c}^{(k)}$
		\Statex \textbf{Output:} $\cc{\bf c}$ and $\cc{\Xi}$ (Now, $\cc{s}\in\cc{\SSS}$, where $\cc{\SSS}$ is the spline space associated to $\cc{\Xi}$ and $\cc{\bf c}$ contains the B-spline coefficients of $\cc{s}$)
	\end{algorithmic}
\end{algorithm}

We now describe the main modules inside this algorithm. The super-index $(k)$ refers to the $k$-th iteration of the while loop. Since the algorithm removes one knot per iteration, it also indicates that $k$ knots have been removed up to that time.

\noindent - In line 2, we compute the local indicators $\varepsilon_j^{(0)}$ for each $j\in \{2, \ldots, N-1\}$ defined by 
$$\varepsilon_j^{(0)}:= \EEE{j}(s),$$
where $\EEE{j}(s)$ is characterized by Theorem \ref{T:error formula}. 

\noindent - The goal of lines 7 to 9 is to compute the control points of a new spline belonging to the space that results from removing one knot in the $j_*^{(k)}$-th breakpoint of $\Xi^{(k)}$. We have split this stage in three steps in order to emphasize the local nature of this update:

\begin{itemize}
    \item In line 7 we just select some components of the vector ${\bf c}^{(k)}$. More especifically, let $i_*= p+1+ \sum_{r=2}^{j_*^{(k)}} m_r^{(k)}$, where $m_r^{(k)}$ is the multiplicity of $r$-th breakpoint in $\Xi^{(k)}$ and let $\ell_*= m_{j_*}^{(k)}-1$. Then
$${\bf c}_\loc^{(k)}= \left( c_{i_* -p-1}^{(k)}, \ldots, c_{i_*-\ell_*}^{(k)} \right),$$
where $c_i^{(k)}$ is the $i$-th component of ${\bf c}^{(k)}$. 
\item In line 8, we compute the new local control points ${\bf c}_\loc^{(k+1)}$ as the least squares solution of the following system
$$ E_\loc A_\loc {\bf c}_\loc^{(k+1)}= E_\loc {\bf c}_\loc^{(k)},$$
where $A_\loc, E_\loc$ are defined in \eqref{Matriz Aloc} and \eqref{ej} with $\icero= i_*$, $\ell= \ell_*$ and considering the knots in $\Xi^{(k)}$.

\item 
In line 9, we assemble the new vector of control points ${\bf c}^{(k+1)}$ by replacing the subvector ${\bf c}_{\loc}^{(k)}$ by ${\bf c}_{\loc}^{(k+1)}$. Thus, we have
\begin{equation} \label{E: new coefficient vector}
\begin{cases}
	c_i^{(k+1)}= c_i^{(k)}, \,\ \,\ \,\ i \leq i_* -p-2,\\
	c_i^{(k+1)}= c_{i+1}^{(k)}, \,\ \,\ \,\ i\geq i_* -\ell_*.
 & 
\end{cases}
\end{equation}
\end{itemize}

\noindent - In line 10, we remove the knot $\xi_{i_*}$ from the current knot vector, that is, $\Xi^{(k+1)} := \Xi^{(k)} \setminus \{\xi_{i_*}\}$, so that
\begin{equation} \label{update opne knot vector}
  \begin{cases}
	\xi_i^{(k+1)}= \xi_i^{(k)}, \,\ \,\ \,\ i= 1, \ldots, i_* -1,\\
	\xi_i^{(k+1)}= \xi_{i+1}^{(k)}, \,\ \,\ \,\ i= i_*, \ldots, n+p,& 
\end{cases}
\end{equation}
and we also let $N^{(k+1)}$ be the number of breakpoints of $\Xi^{(k+1)}$, given by $N^{(k+1)}= N^{(k)}$ when $\ell_* > 0$, and $N^{(k+1)}= N^{(k)}-1$ when $\ell_* = 0$.

\noindent - In line 11, we compute the local indicators $\{\varepsilon_j^{(k+1)}\}_{j=2}^{N^{(k+1)}-1}$ as explained below. Such indicators are defined by
\begin{equation}\label{E:epsilon j+1}
    \varepsilon_j^{(k+1)} := \mathbb{E}_{\Xi^{(k+1)}, j}(s_{k+1}),
\end{equation}
where $s_{k+1}$ is the spline function with control points ${\bf c}^{(k+1)}$ in the spline space associated to the knot vector $\Xi^{(k+1)}$. According to Theorem~\ref{T:error formula} and Remark~\ref{R:calculo del error} we have that  $\varepsilon_j^{(k+1)}$ depends only on the knots in  $\Xi_{\loc,i_j}^{(k+1)}:=\{\xi_{i_j -p-1}^{(k+1)}, \ldots, \xi_{i_j +p+1-\ell_{j}}^{(k+1)}\}$ and the coefficients in ${\bf c}_{\loc,i_j}^{(k+1)}:= \{c_{i_j -p-1}^{(k+1)}, \ldots, c_{i_j-\ell_{j}}^{(k+1)}\}$. Here, $i_j := p+1+\sum_{r=2}^j m_r^{(k+1)}$, where  $m_r^{(k+1)}$ denotes the multiplicity of $r$-th breakpoint in $\Xi^{(k+1)}$ and $\ell_{j}= m_{j}^{(k+1)}-1$.
Thus, taking into account \eqref{E: new coefficient vector} and \eqref{update opne knot vector} we have that\footnote{Notice that if $\ell_*\ge 1$, then $m_j^{(k+1)}= m_j^{(k)}$, for all $j\neq j_*$. When $\ell_*=0$, we have that $m_j^{(k+1)}= m_j^{(k)}$, for $j<j_*$ and $m_j^{(k+1)}= m_{j+1}^{(k)}$, for $j\ge j_*$.} 

\begin{itemize}
    \item if $i_j \leq i_* -p-2$, $\Xi_{\loc,i_j}^{(k+1)}=\Xi_{\loc,i_j}^{(k)}$ and ${\bf c}_{\loc,i_j}^{(k+1)}={\bf c}_{\loc,i_j}^{(k)}$, whence $\varepsilon_j^{(k+1)} = \varepsilon_j^{(k)}$. 
    \item if $i_j \geq i_* +p+1$, $\Xi_{\loc,i_j}^{(k+1)}=\Xi_{\loc,i_j+1}^{(k)}$ and ${\bf c}_{\loc,i_j}^{(k+1)}={\bf c}_{\loc,i_j+1}^{(k)}$, so that
     $$\varepsilon_j^{(k+1)} =  \begin{cases}
	\varepsilon_{j+1}^{(k)}, \,\ \,\ \,\ \text{if} \,\ \ell_* = 0,\\
	\varepsilon_j^{(k)}, \,\ \,\ \,\ \text{if} \,\ \ell_*\geq 1.& 
\end{cases}$$
\end{itemize}
Therefore, we have to compute $\varepsilon_j^{(k+1)}$ using~\eqref{E:epsilon j+1} only for the few indices $j$ such that $i_* -p-1 \leq i_j \leq i_* +p$.

We conclude this section with the following bound for the discrepancy between the original spline $s$ and the output of  Algorithm~\ref{alg:L2adaptive knots removal}.
\begin{theorem}\label{thm:L2}
	\rm Let $s\in \SSS$ and $\Tol>0$. Then, Algorithm~\ref{alg:L2adaptive knots removal} finishes after a finite number of iterations and returns a spline $\cc{s}\in \cc{\SSS}$ such that
	\begin{equation*} \label{Eq del T del Primer algoritmo}
		\| s -\cc{s} \|_{L^2}\le \| s -\cc{s} \|_{\Xi}< \Tol.
	\end{equation*}
\end{theorem}

\begin{proof}
	It is clear that Algorithm~\ref{alg:L2adaptive knots removal} finishes after $K\leq k_{\max}$ iterations. 
 If $s_{k}$ is the spline function with control points ${\bf c}^{(k)}$ in the spline space associated to the knot vector $\Xi^{(k)}$ we have that
$$\varepsilon^{(k)}:= \varepsilon_{j_*^{(k)}}^{(k)}= \| s_{k} - s_{k+1}\|_{\Xi^{(k)}}, \,\ k=0, \ldots, K-1.$$
	Moreover, $\cc{\Xi}= \Xi^{(K)}$ and $\cc{s}=s_K\in\cc{\SSS}$, where $\cc{\SSS}$ is the spline space associated to $\cc{\Xi}$.
Now, Theorem~\ref{T:estability} yields
$$\| s -\cc{s} \|_{L^2}\le \| s -\cc{s} \|_{\Xi} =  \| s_0 -s_K \|_{\Xi^{(0)}} \leq \sum_{k=0}^{K-1} \| s_k-s_{k+1}\|_{\Xi^{(0)}}.$$
Since $\Xi^{(k)}$ is a subsequence of $\Xi^{(0)}$, from~\cite[Proposition 5.2]{T.L&K.M} we obtain that $\| s_k-s_{k+1}\|_{\Xi^{(0)}}\le \varepsilon^{(k)}$. Finally, noticing that the algorithm guarantees that 
$$\varepsilon^{(K-1)}< \Tol- (\varepsilon^{(0)}+ \varepsilon^{(1)}+ \ldots + \varepsilon^{(K-2)}),$$
we conclude that
$$\| s -\cc{s} \|_{L^2}\le \| s -\cc{s} \|_{\Xi} \leq \sum_{k=0}^{K-1} \varepsilon^{(k)}<\Tol,$$
which completes the proof.
\end{proof}

\subsection{Algorithm for the $L^\infty$-norm}

In this section we explain how Algorithm~\ref{alg:L2adaptive knots removal} can be  easily modified in order to obtain an approximation for a given spline in a coarser space such that the distance in ${L^\infty}$-norm is less than a prescribed tolerance.

We denote by $\| \cdot \|_\cpinfnorm$ the spline norm of $s$ defined by
$\|s\|_\cpinfnorm:= \| {\bf c}\|_\infty = \max_{i=1, \ldots,n} |c_i|$, where $c_i$ is the $i$-th B-spline coefficient of $s$. Thus, the $L^\infty$-stability for the B-spline basis (cf.~\eqref{estability of Bsplines in Linfty}) reads:
\begin{equation} \label{equivalencia de normas}
		K_p^{-1} \|s\|_\cpinfnorm \leq \|s\|_{L^\infty} \leq \|s\|_\cpinfnorm, \qquad \forall s\in \SSS.
	\end{equation}
Following the notation introduced at the beginning of Section~\ref{knot removal} we consider 
\begin{equation} \label{Residuo con norma cp}
    \EEE{\jcero}^{\| \cdot \|_\cpinfnorm}(s)= \min_{g\in \cc{\SSS}} \| s-g\|_\cpinfnorm,
\end{equation}
and notice that the spline $g$ achieving the minimum in \eqref{Residuo con norma cp} is, in general, not unique. 
Besides, it is easy to check that the minimum in
 $ \EEE{\jcero}^{\| \cdot \|_\cpinfnorm}(s)= \min_{{\bf z}\in \RR^{p+2-\ell}} \| {\bf c}_\loc - A_\loc {\bf z}\|_\infty$
is achieved at a unique ${\bf \cc{c}}_\loc$ given by
\begin{equation}\label{E:problema local Linf}
    {\bf \cc{c}}_\loc= \argmin_{{\bf z}\in \RR^{p+2-\ell}} \| {\bf c}_\loc - A_\loc {\bf z}\|_\infty.
\end{equation}

 From the analysis in~\cite[Section 4]{T.L&K.M}, it follows that the square matrix $M:=[A_\loc \mid \bf{s}]$ is non-singular, where ${\bf s}= (s_1, \ldots, s_{p+2-\ell})^T$ with $s_i=(-1)^{p-\ell-i}$, for $i =1, \ldots, p+2-\ell$. Moreover, if ${\bf x}= (x_1, \ldots, x_{p+2-\ell})^T$ is the solution of $M{\bf x}= {\bf c}_\loc$, then $\cc{\bf c}_\loc = (x_1, \ldots, x_{p+1-\ell})^T$ and
 $\EEE{\jcero}^{\| \cdot \|_\cpinfnorm}(s)=|x_{p+2-\ell}|$.

Algorithm~\ref{alg:Linf knots removal} is the coarsening algorithm for the $L^\infty$-norm.

\begin{algorithm}[htb]
\setstretch{1.35}
	\caption{\bf  $L^\infty$\_adaptive\_knot\_removal}  \label{alg:Linf knots removal}
	\begin{algorithmic}[1]
\Statex Follow the same lines in Algorithm~\ref{alg:L2adaptive knots removal} taking into account the slight modifications detailed below:

\noindent - In line 2, we compute the local indicators $\varepsilon_j^{(0)}:=\EEE{j}^{\| \cdot\|_\cpinfnorm}(s)$ for each $j\in \{2, \ldots, N-1\}$. 

\noindent - In line 8, we compute ${\bf {c}}_\loc^{(k+1)}$ given by
$$  {\bf {c}}_\loc^{(k+1)}= \argmin_{{\bf z}\in \RR^{p+2-\ell_*}} \| {\bf c}_\loc^{(k)} - A_\loc {\bf z}\|_\infty,$$
as explained above, 
where $A_\loc$ is defined in \eqref{Matriz Aloc} with $\icero= i_*$, $\ell= \ell_*$ and considering the knots in $\Xi^{(k)}$.

\noindent - In line 11, we proceed as in line 11 of Algorithm~\ref{alg:L2adaptive knots removal} and compute only $\varepsilon_j^{(k+1)} := \mathbb{E}^{\cpinfnorm}_{\Xi^{(k+1)}, j}(s_{k+1})$, for $i_* -p-1 \leq i_j \leq i_* +p$.
	\end{algorithmic}
\end{algorithm}

\begin{remark}
    Notice that the explanation of line 11 of Algorithm~\ref{alg:L2adaptive knots removal} also applies in this case because $\varepsilon_j^{(k+1)}$ depends only on the local knot insertion matrix and on the same few control points as before, see~\eqref{E:problema local Linf}.
\end{remark}

We conclude this section with the following result. 

\begin{theorem}
	\rm Let $s\in \SSS$ and $\Tol>0$. Then, Algorithm~\ref{alg:Linf knots removal} finishes after a finite number of iterations and returns a spline $\cc{s}\in \cc{\SSS}$ such that
	\begin{eqnarray*}
		\| s -\cc{s} \|_{L^\infty} < \Tol.
	\end{eqnarray*}
\end{theorem}
\begin{proof}
Similarly to the proof of Theorem~\ref{thm:L2}, Algorithm~\ref{alg:Linf knots removal} finishes after $K\leq k_{\max}$ iterations. 
 If $s_{k}$ is the spline function with control points ${\bf c}^{(k)}$ in the spline space associated to the knot vector $\Xi^{(k)}$ we have that
$$\varepsilon^{(k)}:= \varepsilon^{(k)}_{j^{(k)}_*}= \mathbb{E}_{\Xi^{(k)}, j^{(k)}_*}^{\| \cdot\|_\cpinfnorm}(s_k)= \| s_k-s_{k+1} \|_\cpinfnorm, \,\ k=0, \ldots, K-1.$$
	Now, since $\cc{\Xi}= \Xi^{(K)}$ and $\cc{s}=s_K\in\cc{\SSS}$, where $\cc{\SSS}$ is the spline space associated to $\cc{\Xi}$, and taking into account~\eqref{equivalencia de normas}, we conclude that
 $$      \| s - \cc{s}\|_{L^\infty} 
      = \| s_0 - s_K \|_{L^\infty} 
      \leq \sum_{k=0}^{K-1} \| s_k - s_{k+1} \|_{L^\infty}
      \leq \sum_{k=0}^{K-1} \| s_k - s_{k+1} \|_\cpinfnorm
      \\
      =\sum_{k=0}^{K-1} \varepsilon^{(k)} <  \Tol.
$$
\end{proof}

\subsection{Algorithm for the $H^1$-norm}

We conclude this section by considering the case of the $H^1$-norm. Roughly speaking, given a tolerance $\Tol>0$ and a $C^0$-spline function $s$, we first apply Algorithm~\ref{alg:L2adaptive knots removal} to the right-derivative $s'$ of $s$ to obtain a coarsen approximation $\hat{s}'$ satisfying $\| s' - \hat{s}' \|_{L^2} < \Tol'$, for a suitable value of $\Tol'>0.$
We then integrate the spline $\cc{s}'$ to obtain a spline $\cc{s}$ such that $\| s - \hat{s} \|_{H^1} < \Tol$.

Let $\SSS$ denote the spline space associated to a $(p+1)$-open\footnote{The $(p+1)$-basic knot vector $\Xi$ is called \emph{open} if $\xi_{1}=\dots=\xi_{p+1}$ and $\xi_{n+1}=\dots=\xi_{n+p+1}$.} knot vector $\Xi= \{ \xi_i\}_{i=1}^{n+p+1}$. 
In this section, we assume that the multiplicity of each breakpoint is at most $p$ so that $\SSS\subset C[a,b]$ and $\SSS\subset H^1(a,b)$. 
The set of the right-derivatives, defined by $\SSS':=\{s'\,\mid\, s\in\SSS\}$ can be characterized as the spline space associated to the $p$-open knot vector $\Xi':=\{\xi_i\}_{i=2}^{n+p}$, cf.~\cite[Theorem 7]{Cetraro}.

If ${\bf c} = (c_1,\dots,c_n)^T$ denotes the vector of B-spline coefficients of $s\in\SSS$, it is well known that the vector ${\bf c}' = (c_2',\dots, c_n')^T$ of B-spline coefficients of $s'$ is given by
\begin{equation}\label{coeff de sprima}
    c_{i}'= \left( \frac{c_{i}-c_{i-1}}{\xi_{i}^* - \xi_{i-1}^*} \right), \qquad i=2, \ldots, n,
\end{equation} 
where $\xi_i^*:= \frac{\xi_{i+1} + \ldots + \xi_{i+p}}{p}$ denotes the $i$-th Greville abscissa.

\begin{algorithm}[htb]
\setstretch{1.35}
	\caption{\bf $H^1$\_adaptive\_knot\_removal}  \label{alg:H1adaptive knots removal}
	\begin{algorithmic}[1]
		\Statex \textbf{Input:} $\Tol>0$,  ${\bf c}$ and $\Xi$ (Here, ${\bf c}$ contains the B-spline coefficients of $s\in \SSS$, where $\SSS\subset C[a,b]$ is the spline space associated to a $(p+1)$-open knot vector~$\Xi$) 
		\State $ [{\bf c}',\Xi']\gets {\bf compute\_control\_points\_of\_the\_derivative}({\bf c},\Xi)$ 
  \State $\Tol' = \Tol/\sqrt{(b-a)^2+1}$ 
		\State ${[\bf \cc{c}',\cc\Xi']} \gets \text{\bf $L^2$\_adaptive\_knots\_removal}(\Tol', {\bf c}', \Xi')$ \Comment{Algorithm~\ref{alg:L2adaptive knots removal}} 
  \State Build $\cc{\Xi}$ from  $\cc\Xi'$
    \State $\cc{\bf c} \gets \text{\bf compute\_control\_points\_of\_the\_primitive}( \cc{\bf c}',\cc{\Xi})$ 
    \Statex \textbf{Output:} $\cc{\bf c}$ and $\cc{\Xi}$ (Now, $\cc{s}\in\cc{\SSS}$, where $\cc{\SSS}$ is the spline space associated to $\cc{\Xi}$ and $\cc{\bf c}$ contains the B-spline coefficients of $\cc{s}$)
	\end{algorithmic}
\end{algorithm}

Algorithm~\ref{alg:H1adaptive knots removal} is the coarsening algorithm for the $H^1$-norm, the modules of which we explain in the following paragraphs.

- In line 1 we compute the B-spline coefficients ${\bf c}'$ of $s'$ using~\eqref{coeff de sprima}, and build the corresponding $p$-open knot vector by removing one occurence of the first and the last knot from~$\Xi$.

- In line 3 we apply Algorithm~\ref{alg:L2adaptive knots removal} to obtain a spline $\hat{s}'$ satisfying $\| s' - \hat{s}' \|_{L^2} < \Tol'$, with $\Tol'$ as defined in line 2. Notice that Algorithm~\ref{alg:L2adaptive knots removal} is applied in a space with splines of degree $\le p-1$.  

- In line 4 we build the $(p+1)$-open knot vector $\cc{\Xi}$  obtained from $\cc{\Xi}'$ by adding one time the first and the last knot.

- In line 5 we compute the B-spline coefficients ${\bf \cc{c}}= (\cc{c}_1,\dots, \cc{c}_{\cc{n}})^T$ of $\cc{s}$ as follows
\begin{equation} \label{coeff of f sombrero}
\begin{split}
    \cc{c}_1 &= c_1,\\
    \cc{c}_i &= \cc{c}'_{i} ( \cc{\xi}^*_i - \cc{\xi}^*_{i-1}) + \cc{c}_{i-1} , \,\ \,\ \,\ i= 2, \ldots, \cc{n},
\end{split}    
\end{equation}
where ${\bf \cc{c}'}=(\cc{c}_2',\dots,\cc{c}_{\cc{n}}')^T$ are the B-spline coefficients of $\cc{s}'$, and $\cc{\xi}_i^*$ are the Greville abscissas associated to $\cc{\Xi}$.

At this point, it is important to emphasize that $\cc{\Xi}$ can be indeed obtained from $\Xi$ by removing some knots. Additionally, we remark that since $\cc{\Xi}$ and $\Xi$ are \emph{open}, the first equality in~\eqref{coeff of f sombrero} guarantees that $\cc{s}(a)=s(a)$.

Regarding Algorithm~\ref{alg:H1adaptive knots removal}, we have the following result.

\begin{theorem}
	\rm Let $s\in \SSS$ and $\Tol>0$. Then, Algorithm~\ref{alg:H1adaptive knots removal} finishes and returns a spline $\cc{s}\in \cc{\SSS}$ such that
	\begin{eqnarray*}
		\| s -\cc{s} \|_{H^1} < \Tol.
	\end{eqnarray*}
\end{theorem}

\begin{proof}
Due to Theorem~\ref{thm:L2}, we have that $ | s- \cc{s} |_{H^1} < \Tol' = \frac{\Tol}{\sqrt{(b-a)^2+1}}$. Additionally, since $\cc{s}(a)=s(a)$, the following Poincaré inequality holds:
\begin{equation*}
    \| s- \cc{s}\|_{L^2} \leq (b-a) |s-\cc{s} |_{H^1}.
\end{equation*}
Thus,
$$ 
\| s- \cc{s}\|_{H^1}^2 = \| s-\cc{s}\|_{L^2}^2 + |s-\cc{s} |_{H^1}^2 \leq [(b-a)^2+1]|s-\cc{s} |_{H^1}^2 < [(b-a)^2+1]\Tol'^2= \Tol^2,
$$
which concludes the proof.
\end{proof}

\section{Numerical experiments}\label{S:Numerical experiments}

We finish this article with some numerical tests that show the performance of the algorithms proposed in the previous section and briefly illustrate some useful applications to data reduction and local coarsening in numerical methods for partial differential equations.

\begin{example}[Adaptive coarsening and adaptive refinement]

In this test we illustrate the fact that the local coarsening can be regarded as the reverse procedure of local refinement. We consider the Runge function $f_1(x) = \frac{1}{1+x^2}$, for $-5\le x\le 5$, and perform adaptive refinement in order to approximate it considering the $L^2$-projection onto the space of $C^0$ splines of degree $\le p$ defined on the graded meshes. We have considered $p= 2$ and $p= 4$. Then, starting with the finest adaptive mesh and the corresponding spline that best approximates $f_1$ we apply Algorithm~\ref{alg:L2adaptive knots removal} and compute the $L^2$-error after each knot removal. The results obtained with four strategies are presented in Figure~\ref{F:coarsening ex1_1}. Strategy 1 consists in applying Algorithm~\ref{alg:L2adaptive knots removal} as stated in the previous section whereas the other strategies consider different error indicators.
Strategies 2, 3 and 4 consider $\EEE{j}^{\| \cdot \|_{\cpnorm}}(s)$ given in~\eqref{E:error en norma sin peso}, $|D|$ (see Remark~\ref{R:Other indicators}) and the jumps (cf.~Theorems~\ref{T:jump formula} and~\ref{T:error y salto}) as local indicators, respectively.
Strategies 1 to 3 show optimal slopes for the error in terms of degrees of freedom. 
The behavior of strategy 4 is very poor, and we thus disregard it in the subsequent experiments.

\begin{figure}[h]
	\begin{center}
		\includegraphics[ width=.48\textwidth]{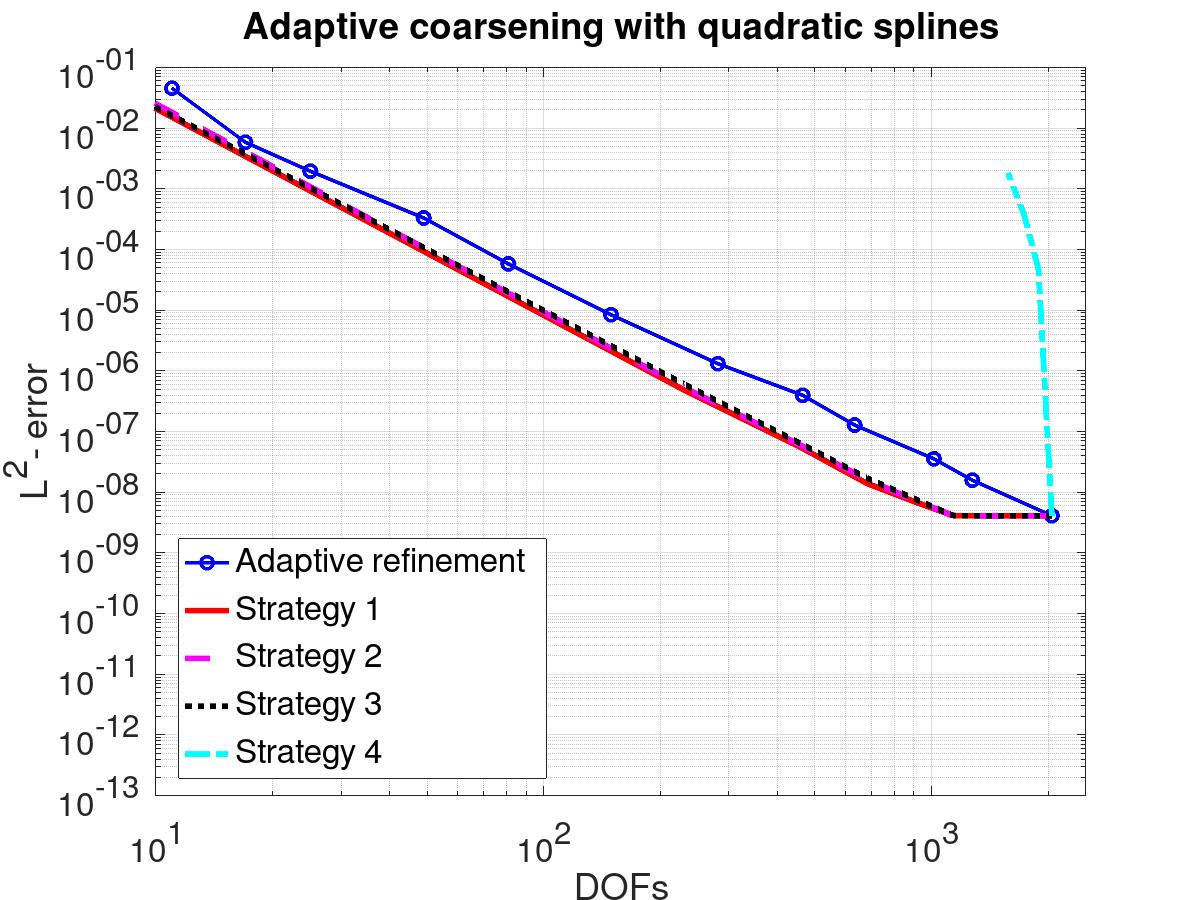}\hfill\includegraphics[width=.48\textwidth]{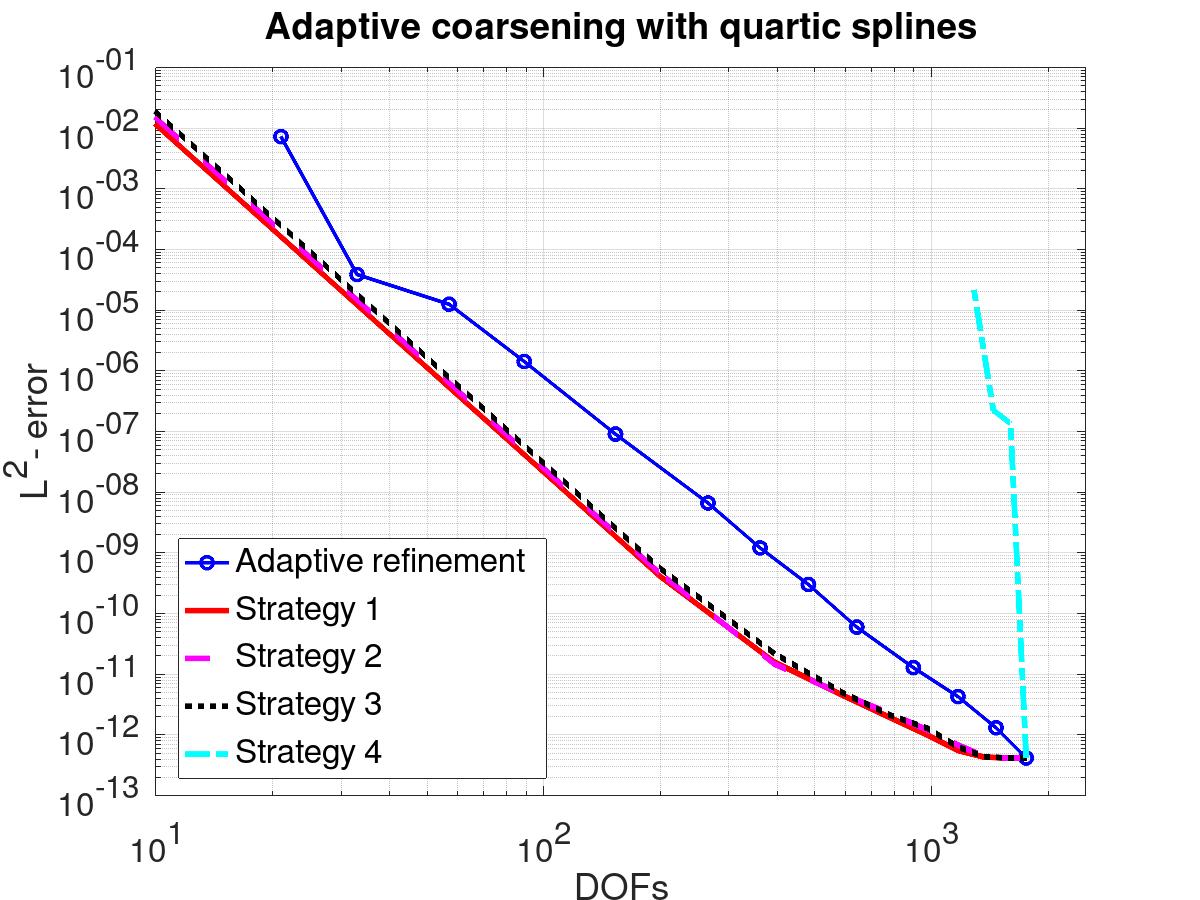}
		\caption{Adaptive coarsening for $f_1(x) = \frac{1}{1+x^2}$, for $-5\le x\le 5$. Strategies 1 to 3, using $\EEE{j}(s)$, $\EEE{j}^{\| \cdot \|_{\cpnorm}}(s)$ and $|D|$ as local indicators, respectively, show optimal slopes for the error in terms of degrees of freedom. The behavior of strategy 4, using the jumps as local indicators, is very poor.} 
		\label{F:coarsening ex1_1}
	\end{center}
\end{figure}

Since Strategies 1 and 2 correspond to solving minimization problems in the~\emph{localized} norms $\|\,\cdot\|_{\Xi}$ and $\|\,\cdot\|_{\cpnorm}$, respectively, the choice of the control points vector $\bf\cc{c}$ after each knot removal is clear, inexpensive, and explicitly stated in Algorithm~\ref{alg:L2adaptive knots removal}. The fact that only a few control points change after each iteration implies that the update of the error indicators is also very cheap because most of them remain unchanged. In this respect, the only drawback of Strategy 2 is that the norm $\|\,\cdot\|_{\cpnorm}$ is not equivalent to the $L^2$-norm. Regarding Strategies 3 and 4, it is not clear how the control points should be defined after each knot removal. In order to benefit these strategies we computed $\cc{s}$ as the $L^2$-projection of $s$ after each iteration (by solving a global linear system). This is not so convenient computationally, as compared to the first two strategies.

In Figure~\ref{F:coarsening ex1_2} we show the error curves for Strategies 1 to 3 applied to $f_2(x)=\sqrt[5]{x}$ on the interval $[-1,1]$. It is notorious in this example that Strategy 1 outperforms the other two. This is to be expected because the former is the only one that is guaranteed to keep an approximation below a given tolerance in $L^2$. Strategies 2 and 3 perform well, but are sub-optimal, and no theory guarantees to keep the $L^2$-error under control.

Thus, we conclude that our algorithm automatically selects the most convenient knots to be sequentially removed. 

\begin{figure}[h!tbp]
	\begin{center}
		\includegraphics[ width=.48\textwidth]{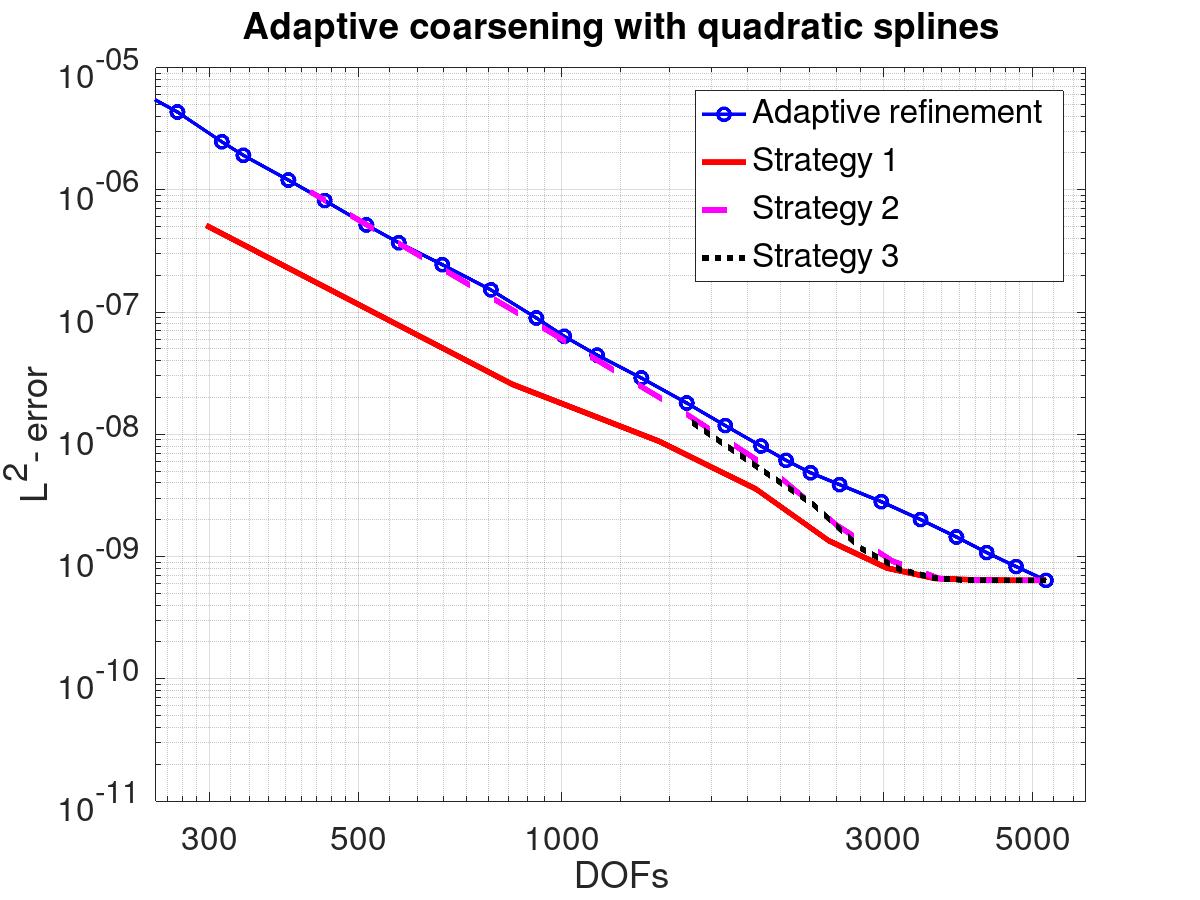}\hfill\includegraphics[width=.48\textwidth]{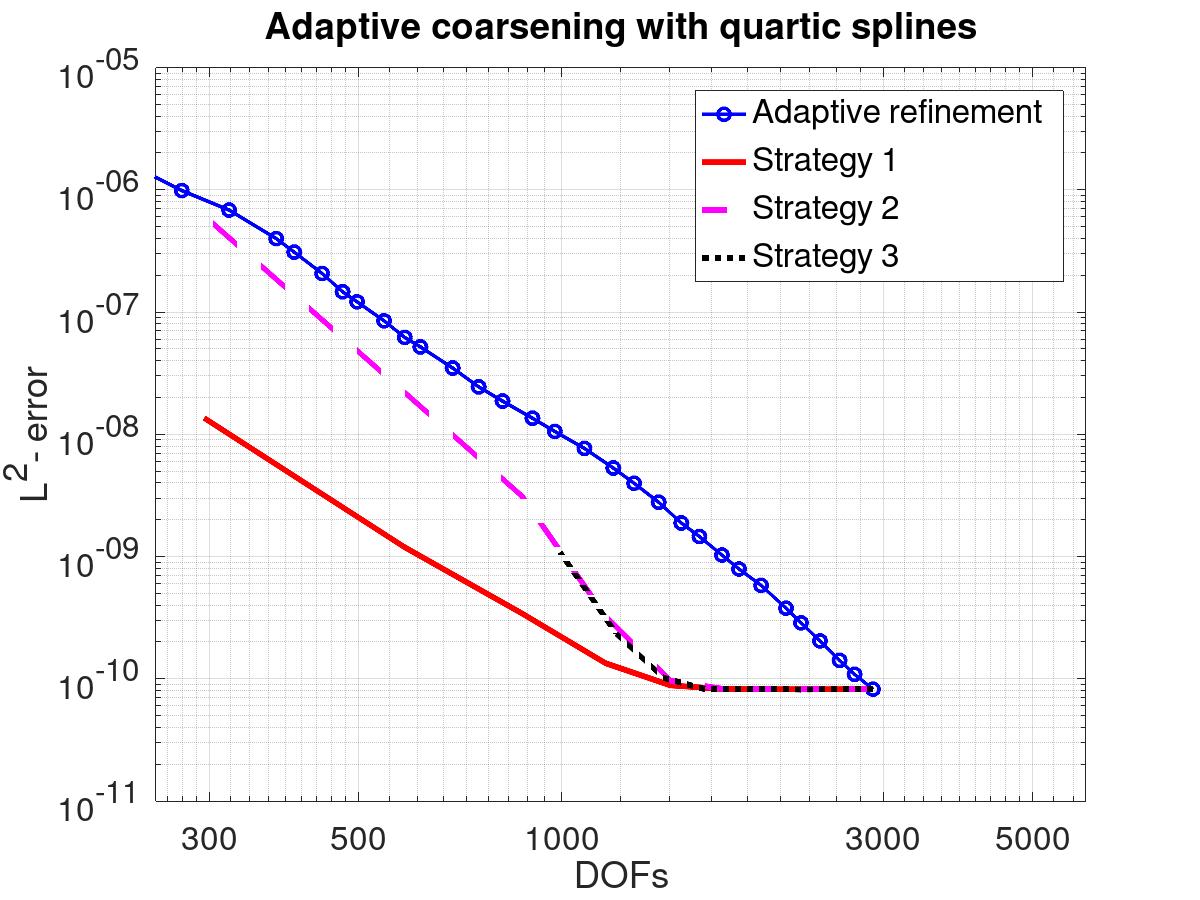}
		\caption{Adaptive coarsening for $f_2(x) = \sqrt[5]{x}$, for $-1\le x\le 1$.
                 It is notorious in this example that Strategy 1 outperforms the other two. This is to be expected because the former is the only one that is guaranteed to keep an approximation below a given tolerance in $L^2$. Strategies 2 and 3 perform well, but are sub-optimal, and no theory guarantees to keep the $L^2$-error under control.
            } 
		\label{F:coarsening ex1_2}
	\end{center}
\end{figure}

\end{example}

\begin{example}[Data approximation in maximum-norm] We consider an example inspired by~\cite[Example 6.3]{T.L&K.M}. We sample the Runge function 
at 101 equally spaced points over the interval $[-5,5]$. 
We consider the continuous piecewise linear interpolant to the data written as a linear combination of $C^0$ cubic B-splines. We then apply Algorithm~\ref{alg:Linf knots removal} to remove knots until only 7 and 3 interior knots remain, respectively. In Figure~\ref{F:Linfty} we show these approximations, and the error with respect to the original sampling. The results  obtained here by simply applying Algorithm~\ref{alg:L2adaptive knots removal} are similar to those from~\cite[example 6.3]{T.L&K.M} where a multistage data reduction technique has been used.
 
\begin{figure}[h]
	\begin{center}
		\includegraphics[ width=.48\textwidth]{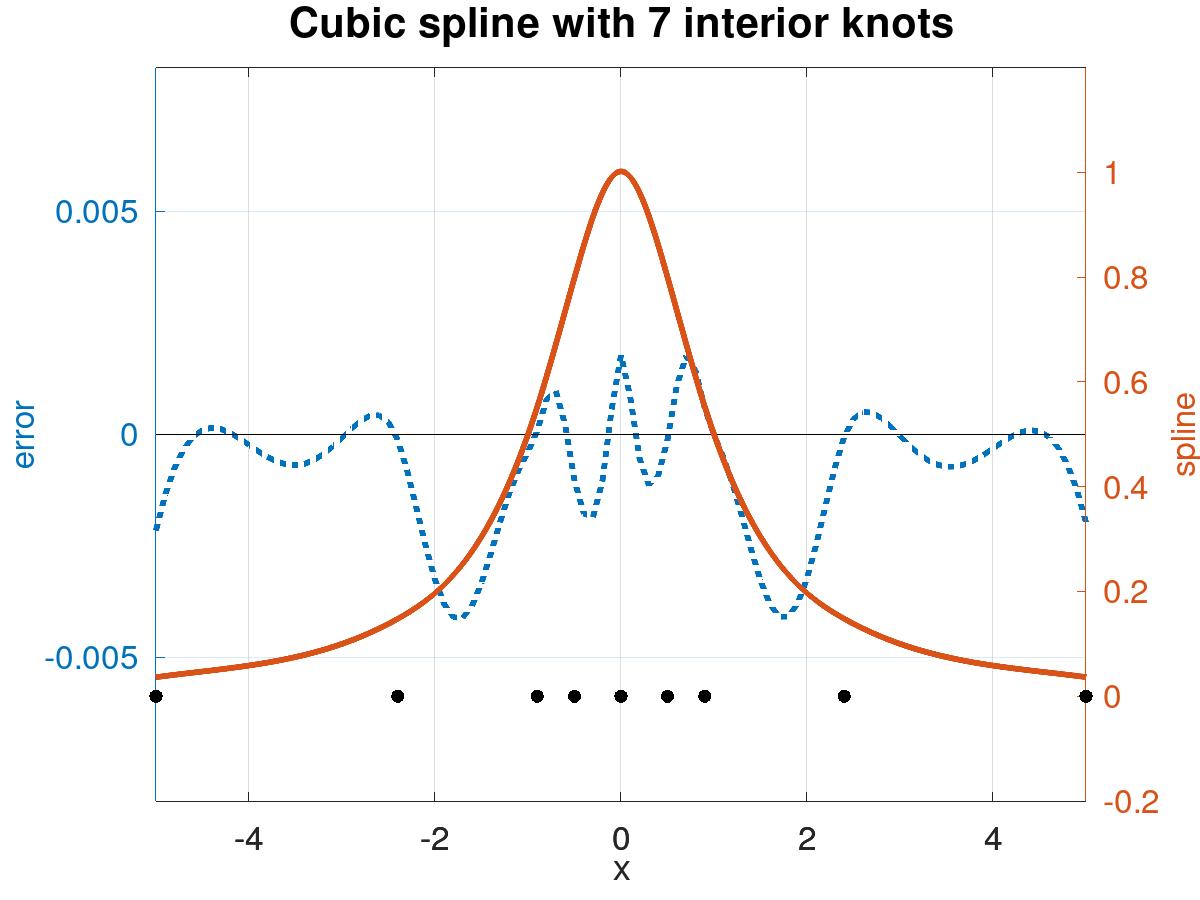}\hfill\includegraphics[width=.48\textwidth]{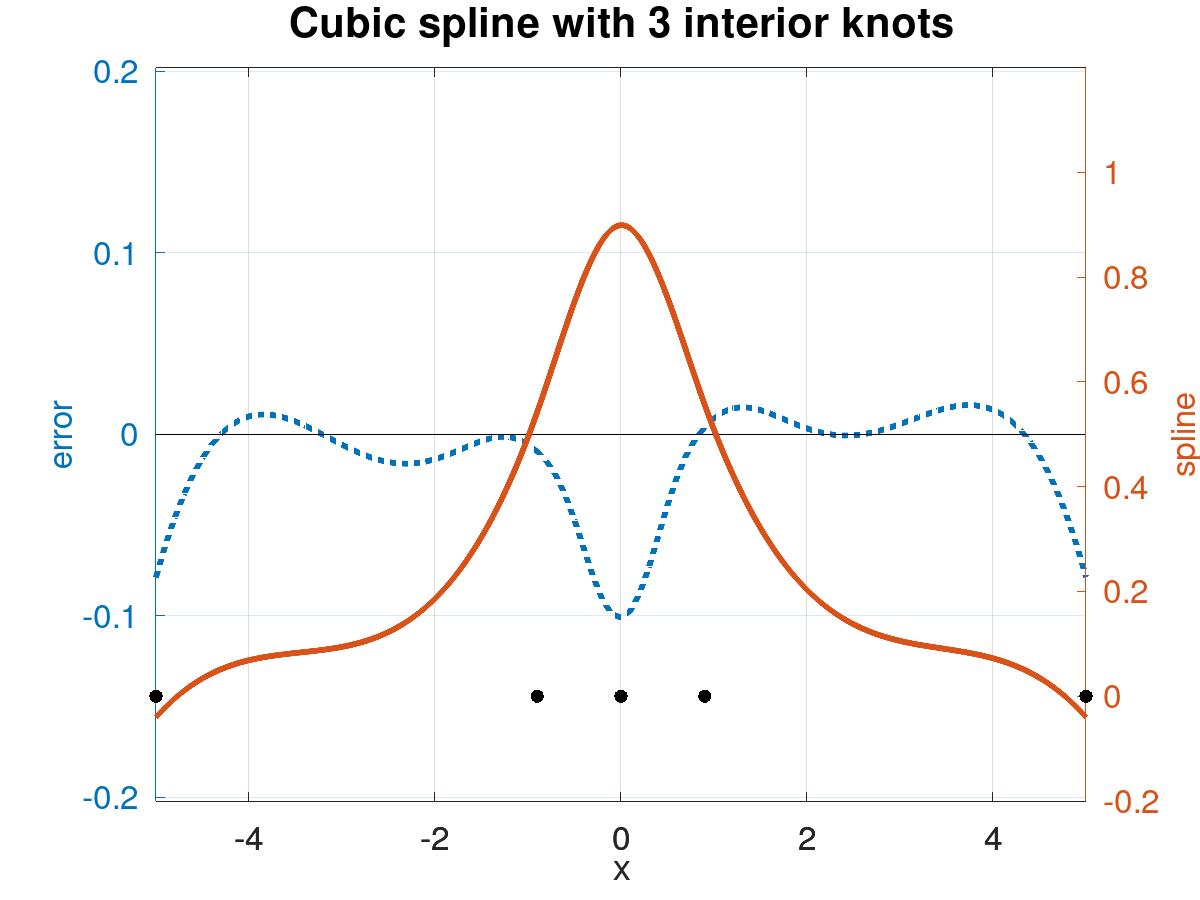}
		\caption{Coarsening up to seven (left) and three (right) interior knots after starting with 297 interior knots approximating a sample of the Runge function. The solid line corresponds to the approximating spline and the dotted line to the error at the sampled points.} 
		\label{F:Linfty}
	\end{center}
\end{figure}
\end{example}

\begin{example}[Heat equation with local adaptive coarsening]

We consider the heat equation $\frac{\partial u}{\partial t}-\frac{\partial^2 u}{\partial x^2} = 0$, for $0<x<10$ and $t>0$. We impose homogeneous Neumann boundary conditions for $t>0$ and the initial value $u(x,0)=u_0(x):= 1+\sin(x^{\frac{7}{20}} \exp(\frac{11x}{50}))$, for $0<x<10$, as showed in Figure~\ref{F:initial value} (left). 

For the spatial discretization we consider the space $\SSS= \SSS_p$ of splines of degree $\le p$ of maximum smoothness defined on a partition $Z$ of $1001$ breakpoints equally distributed in the interval $[0,10]$.  The approximation of the initial value is taken to be the $L^2$-projection $s_0\in\SSS_p$ of $u_0$. The initial error $\|s_0-u_0\|_{L^2(0,10)}$ is around $10^{-4}$ when considering the different values of $p=2,3,4$. We have used the assembly routines from~\cite{geopdes}.

For the time discretization we consider the Backward Euler method with time step size $\Delta t = 0.01$. For each fixed polynomial degree $p$, we proceed iteratively in two different ways to get an approximation of the solution $u(x,t)$ at time $t=1$. On the one hand, we consider the the same space $\SSS_p$ at each time step $t_k:=k\Delta t$,  and proceed as usual: we solve a discrete elliptic equation for computing the approximation $s_{k+1}$ of the solution at time $t_{k+1}$ provided the approximation $s_k$ at time $t_k$ is known. On the other hand, as an alternative procedure, we perform a coarsening of $s_k$ before computing the approximation $s_{k+1}$. More specifically, assuming that $s_k$ belongs to a spline space $\SSS_p^{(k)}$ we apply Algorithm~\ref{alg:H1adaptive knots removal} to find a spline space $\SSS_p^{(k+1)}\subset \SSS_p^{(k)}$ and $\cc{s}_k\in \SSS_p^{(k+1)}$ such that $\|\cc{s}_k-s_k\|_{H^1(0,10)}<10^{-3}$. Then, we use $\cc{s}_k$ to find $s_{k+1}$ belonging to $\SSS_p^{(k+1)}$. 

The numerical results show that the quality of both approximations of the solution $u(x,t)$ at time $t=1$ is similar, because we obtain that the difference in $L^2$-norm is around $10^{-4}$, in all cases ($p=2,3,4$). In Figure~\ref{F:initial value} (right) we show curves showing the number of degrees of freedom as a function of time for the different polynomial degrees considered.
In Figure~\ref{F:solution at time 1} we plot the solution at time $t=1$ that was obtained for the different polynomial degrees, showing also the breakpoints of the spline space at this final time.

\begin{figure}[h!tbp]
	\begin{center}
		\includegraphics[ width=.48\textwidth]{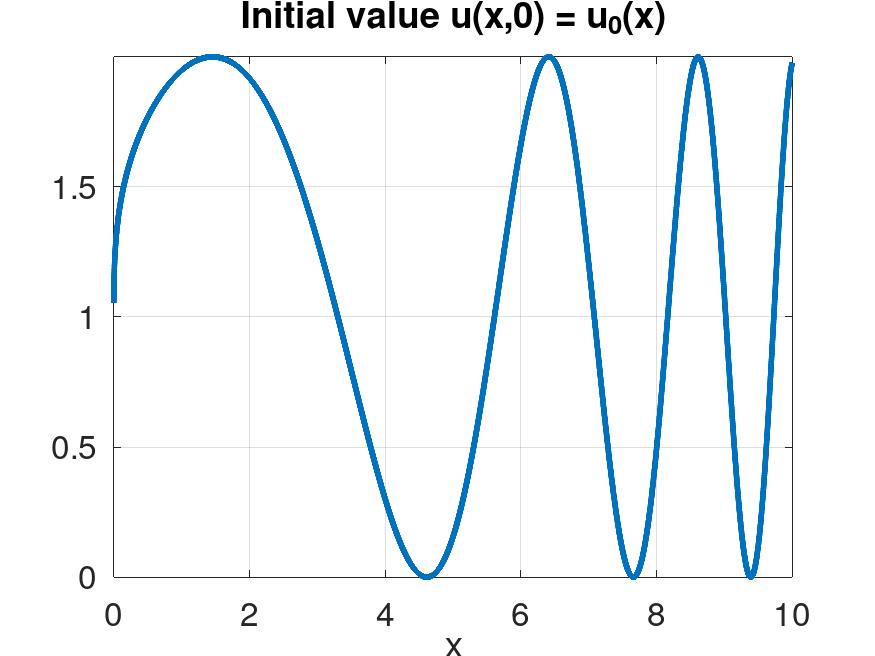}\hfill\includegraphics[width=.48\textwidth]{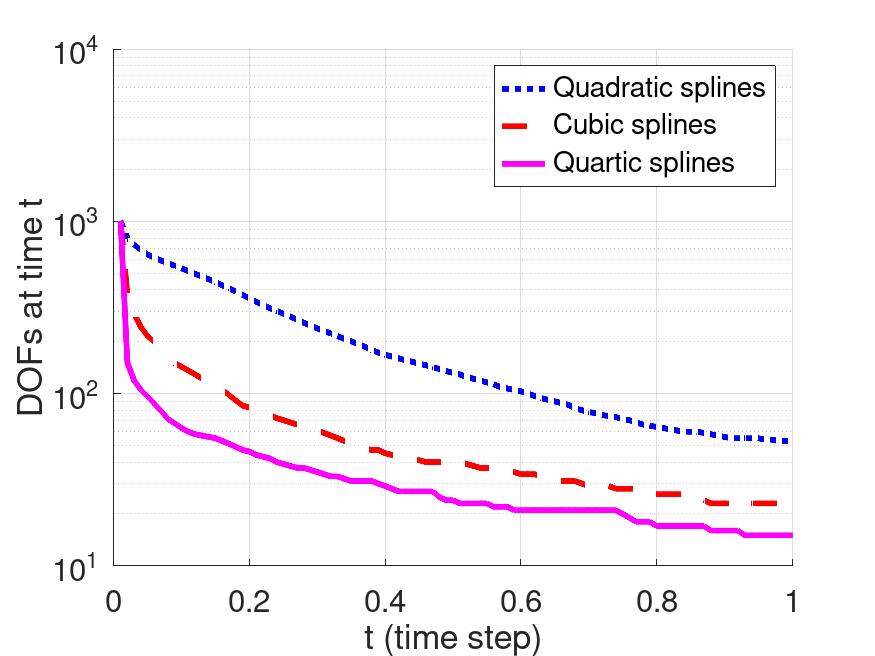}
		\caption{Initial data for the heat equation (left) and curves (right) showing the number of degrees of freedom as a function of time for the different polynomial degrees considered.} 
		\label{F:initial value}
	\end{center}
\end{figure}

\begin{figure}[h!tbp]
	\begin{center}
		\includegraphics[ width=.33\textwidth]{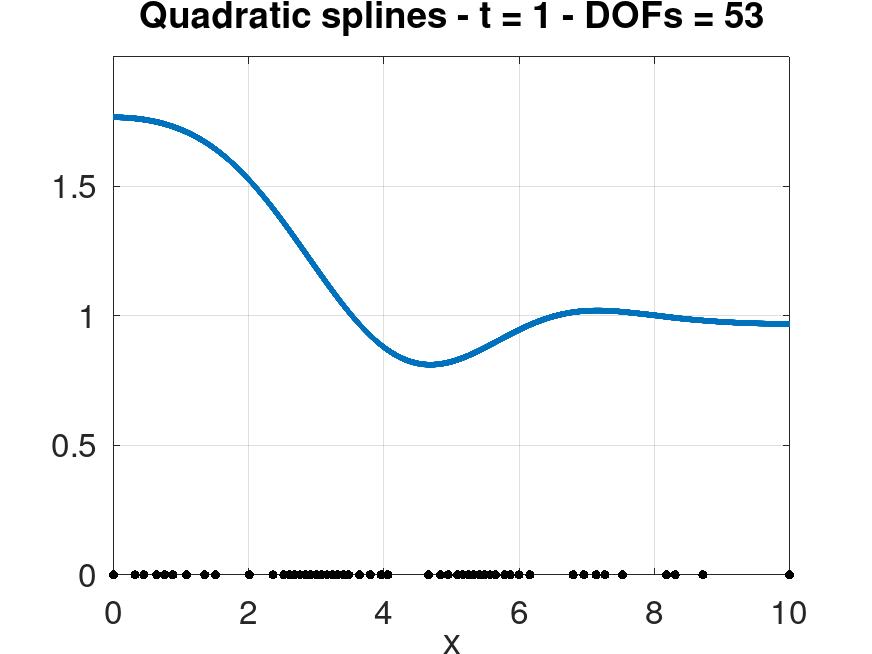}\hfill\includegraphics[width=.33\textwidth]{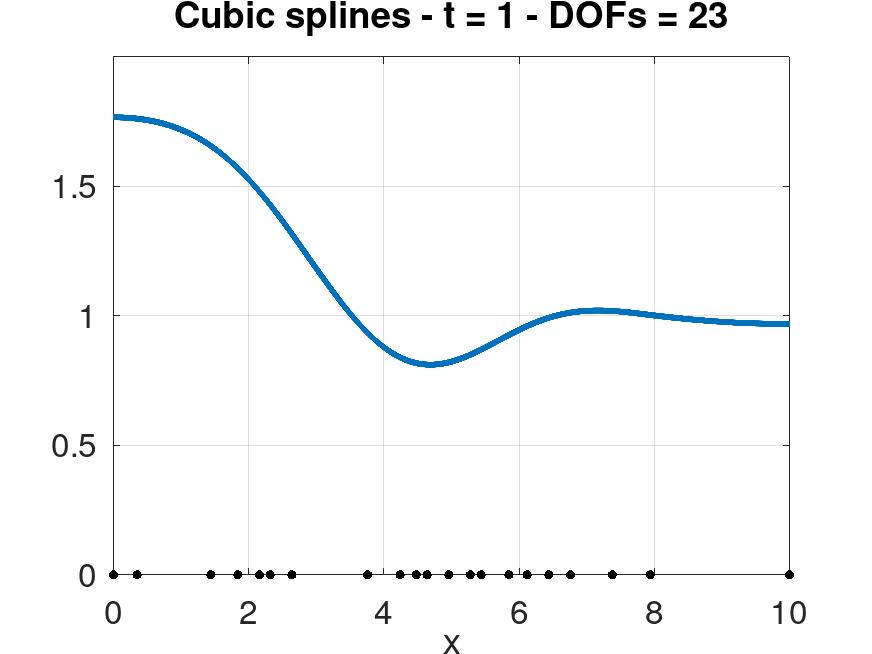}\hfill
  \includegraphics[ width=.33\textwidth]{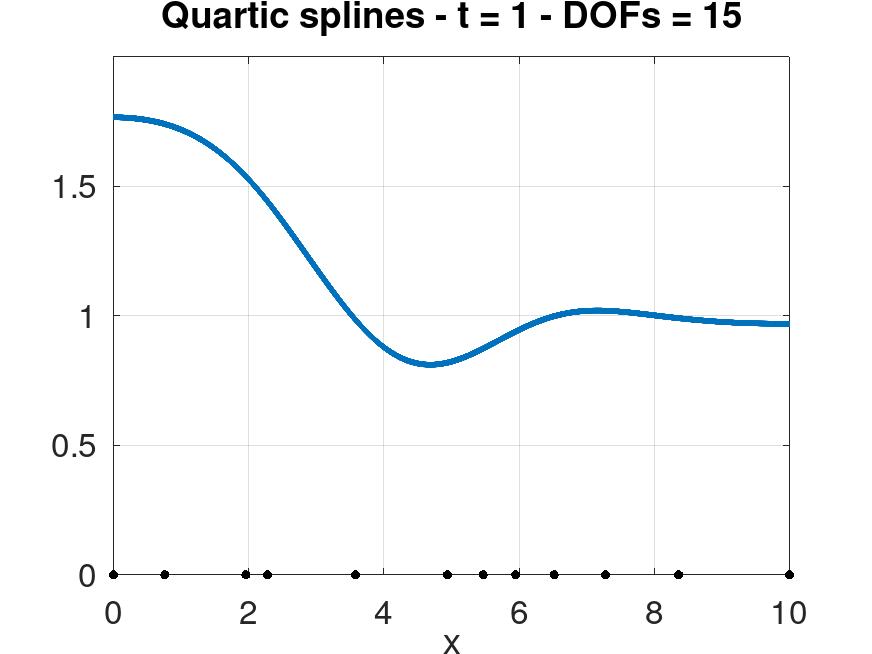}
		\caption{Solution at time $t=1$ obtained for the different polynomial degrees. The dots in the $x$-axis indicate the breakpoints of the spline space at this final time.} 
		\label{F:solution at time 1}
	\end{center}
\end{figure}
    
\end{example}

\bibliographystyle{plain}
\bibliography{biblio}

\end{document}